\documentclass[11p]{article}
\usepackage[a4paper, total={5.5in, 9in}]{geometry}
\usepackage[leqno]{amsmath}
\usepackage{tikz}
\makeatletter
\newcommand{\leqnomode}{\tagsleft@true\let\veqno\@@leqno}
\newcommand{\reqnomode}{\tagsleft@false\let\veqno\@@eqno}
\makeatother
\usepackage{amssymb}
\usepackage{amsthm}
\usepackage{float}
\usepackage{bbm}
\usepackage{color}
\usepackage{multirow}
\usepackage{booktabs}
\usepackage{algpseudocode}
\usepackage{algorithm}
\usepackage{hyperref}
\usepackage{mathtools}
\usepackage{blindtext}

\newtheorem{theorem}{Theorem}
\newtheorem{proposition}{Proposition}
\theoremstyle{definition}
\theoremstyle{definition}
\newtheorem{remark}{Remark}
\theoremstyle{theorm}

\newtheorem{definition}{Definition}
\allowdisplaybreaks

\begin{document}

\title{The Quadratic Cycle Cover Problem: special  cases and efficient bounds}
\author{Frank de Meijer \thanks{CentER, Department of Econometrics and OR, Tilburg University, The Netherlands, {\tt f.j.j.demeijer@uvt.nl}}
	\and {Renata Sotirov}  \thanks{Department of Econometrics and OR, Tilburg University, The Netherlands, {\tt r.sotirov@uvt.nl}}}

\date{}

\maketitle

\begin{abstract}
The quadratic cycle cover problem is the problem of finding a set of node-disjoint cycles visiting
all the nodes such that the total sum of interaction costs between consecutive arcs is minimized.
In this paper we study  the linearization problem for the  quadratic cycle cover problem  and related lower bounds.

In particular, we derive various sufficient conditions for the quadratic cost matrix to be linearizable,
and use these conditions to compute  bounds.
We also show how to use a sufficient condition for linearizability within an iterative bounding procedure.
In each step, our algorithm computes the best equivalent representation of the quadratic cost matrix and
its optimal linearizable matrix with respect to the given sufficient condition for linearizability.
Further, we show that the classical Gilmore-Lawler type bound belongs to the family of linearization based bounds,
and therefore  apply the  above mentioned iterative reformulation technique.
We also prove that the linearization vectors resulting from this iterative approach satisfy the constant value property.

The best among here introduced bounds outperform existing lower bounds when taking both quality and efficiency into account.
\end{abstract}


\reqnomode
\section{Introduction}
A {disjoint} cycle cover in a directed graph is a set of node-disjoint cycles such that every node is on exactly one cycle. The quadratic cycle cover problem (\textsc{QCCP}) is the problem of finding a {disjoint} cycle cover in a graph such that the total sum of interaction costs between {consecutive} arcs is minimized. Since we assume that all cycle covers in this paper are disjoint, we use the term cycle cover to denote this concept throughout this work. The QCCP is proven to be $\mathcal{NP}$-hard {\cite{FischerEtAl}}. The corresponding linear problem is called the cycle cover problem (\textsc{CCP}), in which one wants to find a minimum cycle cover with respect to linear arc costs. It is well known that the \textsc{CCP} is solvable in polynomial time.

In the literature several special cases with respect to the objective function of the \textsc{QCCP} are considered. In the angular metric cycle cover problem (\textsc{Angle-CCP}) the quadratic costs represent the change of the direction induced by two consecutive arcs. The goal of \textsc{Angle-CCP} is to find a cycle cover of the graph while minimizing the total angular costs. The \textsc{Angle-CCP} has applications in robotics \cite{Aggarwal}. In the same paper it is shown that \textsc{Angle-CCP} is $\mathcal{NP}$-hard. Only recently Galbiati et al.\ \cite{Galbiati} introduced another special case of the \textsc{QCCP}: the minimum reload cost cycle cover problem (\textsc{MinRC3}). The \textsc{MinRC3} problem asks for a minimum cycle cover in an arc-colored graph under the reload cost model. A reload cost is an interaction cost that is paid when two arcs of different colors are placed in succession on a cycle. The goal of the \textsc{MinRC3} problem is to find a cycle cover such that the total reload cost is minimized. The problem is proven to be $\mathcal{NP}$-hard in the strong sense \cite{Galbiati}. The notion of reload costs is introduced by Wirth and Steffan \cite{WirthSteffan}, and it has many applications in various fields, e.g. in cargo, energy and telecommunication networks \cite{Amaldi, WirthSteffan}. A detailed overview of the \textsc{MinRC3} problem and its applications can be found in \cite{Buyukcolak}. Several other combinatorial optimization problems including these reload costs have been investigated \cite{Amaldi, Galbiati2008, Gamvros, Gourves, WirthSteffan}.

The \textsc{QCCP} is closely related to the quadratic traveling salesman problem (\textsc{QTSP}) which is introduced in \cite{Jager}. The \textsc{QTSP} is the problem of finding a Hamiltonian cycle in a graph minimizing a quadratic cost function. It has applications in bioinformatics, robotics and telecommunication \cite{AandFFischer}. When we remove the subtour elimination constraints, the \textsc{QTSP} boils down to the \textsc{QCCP}. Therefore, the \textsc{QCCP} is often used to provide lower bounds for the \textsc{QTSP} \cite{AandFFischer, Jager, Stanek}. For this reason, the quadratic cycle cover problem is an interesting optimization problem that has received more attention in the past few years.

Several papers have been written about solution methods for the \textsc{QCCP} or its related problems. J\"ager and Molitor \cite{Jager} introduced the \textsc{QCCP} in order to use the \textsc{QCCP} bounds as lower bounds in a branch-and-bound algorithm for the \textsc{QTSP}.  Stan\v ek et al.\ \cite{Stanek} use the \textsc{QCCP} in combination with a rounding procedure to construct heuristics for the \textsc{QTSP}. Aggarwal et al.\ \cite{Aggarwal} provide a $\mathcal{O}(\log n)$-approximation algorithm for the \textsc{Angle-CCP}. Fischer \cite{Fischer} studies the polyhedral properties of the \textsc{QCCP} by proving that some triangle inequalities are facet-defining. Galbiati et al.\ \cite{Galbiati} derive various integer programming formulations for the \textsc{MinRC3} problem.
They exploit one of those formulations together with a column generation approach to compute  lower bounds for the original problem.
Moreover, in \cite{Galbiati} a local search algorithm based on 2-exchange and 3-exchange neighbourhoods is constructed to compute upper bounds for the \textsc{MinRC3} problem. B\"uy\"uk\c colak et al.\ \cite{Buyukcolak} study the \textsc{MinRC3} problem on complete graphs with an equitable or nearly equitable 2-edge coloring. For these types of graphs (except some special cases) a polynomial time algorithm is derived that constructs a monochromatic cycle cover.

We focus here on the linearization problem of the \textsc{QCCP} and its applications. An instance of the quadratic cycle cover problem is called linearizable if there exists an instance of the linear cycle cover problem such that the associated costs for both problems are equal for all feasible cycle covers. The linearization problem of the quadratic cycle cover problem asks whether a given instance of the \textsc{QCCP} is linearizable. To the best of our knowledge, this is the first paper about the linearization problem of the \textsc{QCCP}.

In the past few years linearization problems have become an active field of research for many combinatorial optimization problems. In \cite{KabadiPunnen1,KabadiPunnen2} the linearization problem of the quadratic assignment problem (QAP) is studied and polynomial time algorithms that solve it are provided.
 In particular, Kabadi and Punnen \cite{KabadiPunnen1}  (resp.~Punnen and Kabadi \cite{KabadiPunnen2}) present an ${\mathcal{O}}(n^4)$  (resp.~${\mathcal{O}}(n^2)$)  algorithm for the  general
 (resp.~Koopmans-Beckmann) QAP linearization problem, where $n$ is the size of the problem.
 The linearization problem for the quadratic minimum spanning tree problem and the quadratic traveling salesman problem are studied by \'Custi\'c and Punnen \cite{CusticPunnen} and Punnen et al.\ \cite{Punnen}, respectively. Hu and Sotirov \cite{HuSotirov2} develop a polynomial time algorithm that solves the linearization problem of the quadratic shortest path problem  on directed graphs.

\paragraph{Main results and outline.}
In this paper, we first  provide an elegant and compact proof that the quadratic cycle cover problem is strongly $\mathcal{NP}$-hard and not approximable within any constant factor unless $\mathcal{P}=\mathcal{NP}$.
Then, we consider the linearization problem of the \textsc{QCCP} and derive various {sufficient} conditions for an instance of the \textsc{QCCP} to be linearizable.
In particular, we provide three different types of weak sum conditions on the data matrix for which the corresponding instance can be solved in polynomial time.
Further, we present a general framework in which each sufficient condition of linearizability can be used to construct a lower bound on the optimal objective value.
Each of these bounds can be computed by a solving a linear programming problem, as long as the set of linearizable matrices is a polyhedron.
These types of bounds are called linearization based bounds (LBB), and were recently introduced in \cite{HuSotirov2} for general binary quadratic problems.
However, our LBBs exploit  sufficient conditions of linearizability  suited for the \textsc{QCCP}.

Furthermore, we show how to use a sufficient condition of linearizability within an iterative bounding procedure.
In each iteration, we search for the best equivalent representation of the objective and its optimal linearizable matrix
that satisfies a particular sufficient condition of linearizability.
We refer to the resulting bound as the reformulation based bound (RBB).
Our iterative bounding procedure can be seen as a generalization  of similar iterative procedures,   see e.g., \cite{Carraresi, RostamiMalucelli, Rostami}.

Finally, we consider the classical Gilmore-Lawler (GL) type bound \cite{Gilmore, Lawler}.  First, we show that the GL type bound for the \textsc{QCCP} can be obtained  by solving
a single linear programming problem instead of solving $m$ (integer) subproblems, where $m$ equals the number of arcs in the graph.
Then, we prove that the GL type bound belongs to the family of linearization based bounds by providing the appropriate sufficient condition.
We implement our iterative bounding procedure to compute the RBB using the  GL type bound.
In {the} literature, iterative approaches for various problems that are based on the  GL type bounds  use dual variables to obtain bounds, and do not search for
equivalent reformulations that  provide best bounds in each iteration. Clearly, our approach outperforms others {in terms of strength of the bound}.
Another interesting result is that the linearization vectors resulting from this iterative procedure satisfy the constant value property.
Yet another important property for linearizability.

Our numerical results show that the  introduced bounding approaches are efficient and provide strong bounds compared to several methods from the literature.
In particular, our most prominent bound  can be computed within 60 seconds for instances up to 15000 arcs.
Interestingly, the GL type bound that is known to be one of the computationally cheapest bounds for quadratic optimization problems
 cannot be computed on such large instances. \\ \\
This paper is organized as follows. In Section \ref{TheQCCP}, we formally introduce the {QCCP} and prove its $\mathcal{NP}$-hardness.
In Section \ref{SufficientConditions}, the linearization problem for the \textsc{QCCP} is introduced and several sufficient conditions for linearizability are derived.
 The general framework for the computation of the linearization based bounds is discussed in Section \ref{LinearizationBasedBounds}.
 These bounds are used in Section \ref{ReformulatedLBBApproach} to construct an iterative bounding procedure for each sufficient condition.
 In Section \ref{GilmoreLawler}, we consider the classical GL type bound and prove that it belongs to the family of linearization based bounds.
 We also show how the iterative procedure for this linearization based bound boils down to the computation of the strongest GL type bound in each step.
 In Section \ref{OtherBounds}, we briefly discuss several other bounds from the literature. Numerical results are given in Section \ref{NumericalResults}.

\subsection*{Notation}
A directed graph $G = (N,A)$ is given by a node set $N$ and an arc set $A \subseteq N \times N$. For all nodes $i \in N$ we denote by $\delta^+(i)$ the set of arcs that are leaving $i$. Similarly, $\delta^-(i)$ denotes the set of arcs that are entering $i$. For all arcs $e \in A$ we let $e^+$ and $e^-$ denote the starting and ending node of $e$, respectively. To avoid confusion, the letters $e,f$ and $g$ are only used to denote arcs in this work.

For any square matrix $M$, we introduce the operator $\text{diag}:\mathbb{R}^{n \times n} \rightarrow \mathbb{R}^n$ that maps a matrix to a vector consisting of its diagonal elements. Moreover, we denote by $\text{Diag}: \mathbb{R}^{n} \rightarrow \mathbb{R}^{n \times n}$ its adjoint operator. That is, for any $v \in \mathbb{R}^n$ the matrix $\text{Diag}(v)$ equals a diagonal matrix with the entries of $v$ on its main diagonal.

\section{The Quadratic Cycle Cover Problem} \label{TheQCCP}
In this section, we formally introduce the quadratic cycle cover problem.\\ \\
An instance $\mathcal{I}$ of the \textsc{QCCP} is specified by the pair $\mathcal{I} = (G, Q)$, where $G = (N,A)$ is a directed graph with $n$ vertices and $m$ arcs and $Q = (Q_{ef}) \in \mathbb{R}^{m \times m}$ is a quadratic cost matrix. The entries in $Q$ are such that $Q_{ef} = 0$ if $f$ is not a successor of $e$. In other words, the quadratic cost of two arcs $e$ and $f$ can be nonzero only if {the starting node of $f$ equals the ending node of $e$}. In case we also consider linear arc costs, i.e.\ we have a cost function $p: A \rightarrow \mathbb{R}$, we can put these arc costs on the diagonal of the quadratic cost matrix. Therefore, we assume that the cost structure of an instance of the \textsc{QCCP} is fully determined by its quadratic cost matrix.  \\
Now let $x \in \{0,1\}^{m}$ be a vector with $x_e = 1$ if arc $e$ belongs to a cycle cover, and 0 otherwise. Then the \textsc{QCCP} can be formulated as
\begin{equation} \label{QCCPdefinition}
\begin{aligned}
\textit{OPT(Q)} := \, \, \, \min \, & \quad  x^TQx \\
\text{s.t.} & \quad  x \in X,
\end{aligned}
\end{equation}
where $X$ denotes the set consisting of all {disjoint} cycle covers in $G$, i.e.
\begin{align}
X := \left \{ x \in \{0,1\}^{m} \, \, | \, \sum_{e \in \delta^+(i)} x_e = \sum_{e \in \delta^-(i)} x_e = 1 \, , \, \forall i \in N \, \right \}. \label{X}
\end{align}
The above set  equals the set of  directed 2-factors in $G$. For the existence of such a directed 2-factor in a directed graph, see e.g. Chiba and Yamashita  \cite{Chiba}.\\ \\
The quadratic cycle cover problem is $\mathcal{NP}$-hard  {\cite{FischerEtAl}}. {Also, the related problems \textsc{Angle-CCP} and the \textsc{MinRC3} problem are shown to be $\mathcal{NP}$-hard \cite{Aggarwal} and strongly $\mathcal{NP}$-hard \cite{Galbiati}, respectively.} We now provide {an alternative reduction that establishes strong $\mathcal{NP}$-hardness} which is based on a reduction from the quadratic assignment problem. We consider the Koopmans-Beckmann form of the \textsc{QAP} {introduced in} \cite{KoopmansBeckmann}. Let $F$ and $P$ be a set of  $n$ facilities and $n$ locations, respectively, $w: F \times F \rightarrow \mathbb{R}$ a weight function and $d: P \times P \rightarrow \mathbb{R}$ a distance function. Without loss of generality, we assume that $d_{ii} = w_{ii} = 0$ for all $i \in \{1, ..., n\}$. Then, we search for a bijection $\pi : F \rightarrow P$ such that $\sum_{i = 1}^n\sum_{j = 1}^nd_{\pi(i) \pi(j)}w_{ij}$ is minimized. The \textsc{QAP} is $\mathcal{NP}$-hard in the strong sense and not approximable within any constant factor \cite{SahniGonzalez}.

\begin{theorem}
The $QCCP$ is $\mathcal{NP}$-hard in the strong sense and  cannot be approximated within a constant factor unless $\mathcal{P}=\mathcal{NP}$.
\end{theorem}
\begin{proof}
Let an instance $\mathcal{I}$ of the \textsc{QAP} be given, i.e., we consider sets $F = \{1, ...,n\}$ and $P = \{1', ..., n'\}$ with $|P| = |F| = n$, functions $w: F \times F \rightarrow \mathbb{R}$ and $d : P \times P \rightarrow \mathbb{R}$ and a positive integer $K$. We create an instance $\mathcal{I}'$ of the \textsc{QCCP}.

For the reduction we create a directed graph $G = (N,A)$ that consists of cells. A cell belongs to a single facility and consists of $n$ nodes, each of them corresponding to an assignment to one of the $n$ locations. For each facility $i \in F$, we define a set of $n-1$ identical cells, which we call a group. The nodes corresponding to the same assignment within a group are placed on a directed cycle. In this way, we obtain $n$ cycles per group, which we call inner cycles. We set the interaction cost between each of the successive arcs within a group to zero for all groups. In Figure \ref{GroupFacility1} the group corresponding to facility 1 is given.

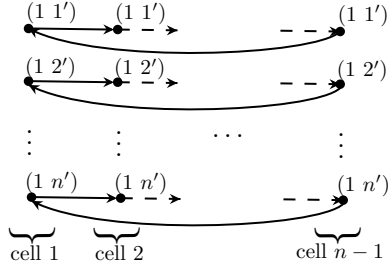
\begin{figure}[H]

\centering

\tikzset{every picture/.style={line width=0.75pt}} %set default line width to 0.75pt

\begin{tikzpicture}[scale = 0.5, x=0.75pt,y=0.75pt,yscale=-1,xscale=1]

%Shape: Circle [id:dp25318656174342724]
\draw  [fill={rgb, 255:red, 0; green, 0; blue, 0 }  ,fill opacity=1 ] (103.64,59.25) .. controls (103.67,61.32) and (102.02,63.03) .. (99.95,63.06) .. controls (97.88,63.09) and (96.17,61.44) .. (96.14,59.37) .. controls (96.11,57.29) and (97.76,55.59) .. (99.83,55.56) .. controls (101.91,55.53) and (103.61,57.18) .. (103.64,59.25) -- cycle ;
%Shape: Circle [id:dp9847072404531916]
\draw  [fill={rgb, 255:red, 0; green, 0; blue, 0 }  ,fill opacity=1 ] (103.64,112.25) .. controls (103.67,114.32) and (102.02,116.03) .. (99.95,116.06) .. controls (97.88,116.09) and (96.17,114.44) .. (96.14,112.37) .. controls (96.11,110.29) and (97.76,108.59) .. (99.83,108.56) .. controls (101.91,108.53) and (103.61,110.18) .. (103.64,112.25) -- cycle ;
%Shape: Circle [id:dp04660566349676243]
\draw  [fill={rgb, 255:red, 0; green, 0; blue, 0 }  ,fill opacity=1 ] (106.64,228.92) .. controls (106.67,230.99) and (105.02,232.69) .. (102.95,232.72) .. controls (100.88,232.76) and (99.17,231.1) .. (99.14,229.03) .. controls (99.11,226.96) and (100.76,225.26) .. (102.83,225.23) .. controls (104.91,225.19) and (106.61,226.85) .. (106.64,228.92) -- cycle ;
%Shape: Circle [id:dp6111024586566245]
\draw  [fill={rgb, 255:red, 0; green, 0; blue, 0 }  ,fill opacity=1 ] (192.64,60.25) .. controls (192.67,62.32) and (191.02,64.03) .. (188.95,64.06) .. controls (186.88,64.09) and (185.17,62.44) .. (185.14,60.37) .. controls (185.11,58.29) and (186.76,56.59) .. (188.83,56.56) .. controls (190.91,56.53) and (192.61,58.18) .. (192.64,60.25) -- cycle ;
%Shape: Circle [id:dp8767051912248269]
\draw  [fill={rgb, 255:red, 0; green, 0; blue, 0 }  ,fill opacity=1 ] (192.64,113.25) .. controls (192.67,115.32) and (191.02,117.03) .. (188.95,117.06) .. controls (186.88,117.09) and (185.17,115.44) .. (185.14,113.37) .. controls (185.11,111.29) and (186.76,109.59) .. (188.83,109.56) .. controls (190.91,109.53) and (192.61,111.18) .. (192.64,113.25) -- cycle ;
%Shape: Circle [id:dp874989347476961]
\draw  [fill={rgb, 255:red, 0; green, 0; blue, 0 }  ,fill opacity=1 ] (195.64,229.92) .. controls (195.67,231.99) and (194.02,233.69) .. (191.95,233.72) .. controls (189.88,233.76) and (188.17,232.1) .. (188.14,230.03) .. controls (188.11,227.96) and (189.76,226.26) .. (191.83,226.23) .. controls (193.91,226.19) and (195.61,227.85) .. (195.64,229.92) -- cycle ;
%Straight Lines [id:da8843602537875905]
\draw    (103.64,59.25) -- (183.14,60.34) ;
\draw [shift={(185.14,60.37)}, rotate = 180.78] [fill={rgb, 255:red, 0; green, 0; blue, 0 }  ][line width=0.75]  [draw opacity=0] (10.72,-5.15) -- (0,0) -- (10.72,5.15) -- (7.12,0) -- cycle    ;

%Shape: Circle [id:dp954654832097922]
\draw  [fill={rgb, 255:red, 0; green, 0; blue, 0 }  ,fill opacity=1 ] (414.64,62.25) .. controls (414.67,64.32) and (413.02,66.03) .. (410.95,66.06) .. controls (408.88,66.09) and (407.17,64.44) .. (407.14,62.37) .. controls (407.11,60.29) and (408.76,58.59) .. (410.83,58.56) .. controls (412.91,58.53) and (414.61,60.18) .. (414.64,62.25) -- cycle ;
%Shape: Circle [id:dp08211727818976788]
\draw  [fill={rgb, 255:red, 0; green, 0; blue, 0 }  ,fill opacity=1 ] (414.64,115.25) .. controls (414.67,117.32) and (413.02,119.03) .. (410.95,119.06) .. controls (408.88,119.09) and (407.17,117.44) .. (407.14,115.37) .. controls (407.11,113.29) and (408.76,111.59) .. (410.83,111.56) .. controls (412.91,111.53) and (414.61,113.18) .. (414.64,115.25) -- cycle ;
%Shape: Circle [id:dp26003129963823524]
\draw  [fill={rgb, 255:red, 0; green, 0; blue, 0 }  ,fill opacity=1 ] (417.64,231.92) .. controls (417.67,233.99) and (416.02,235.69) .. (413.95,235.72) .. controls (411.88,235.76) and (410.17,234.1) .. (410.14,232.03) .. controls (410.11,229.96) and (411.76,228.26) .. (413.83,228.23) .. controls (415.91,228.19) and (417.61,229.85) .. (417.64,231.92) -- cycle ;
%Straight Lines [id:da16271573221340807]
\draw    (103.64,112.25) -- (183.14,113.34) ;
\draw [shift={(185.14,113.37)}, rotate = 180.78] [fill={rgb, 255:red, 0; green, 0; blue, 0 }  ][line width=0.75]  [draw opacity=0] (10.72,-5.15) -- (0,0) -- (10.72,5.15) -- (7.12,0) -- cycle    ;

%Straight Lines [id:da24794729557378803]
\draw    (106.64,228.92) -- (186.14,230) ;
\draw [shift={(188.14,230.03)}, rotate = 180.78] [fill={rgb, 255:red, 0; green, 0; blue, 0 }  ][line width=0.75]  [draw opacity=0] (10.72,-5.15) -- (0,0) -- (10.72,5.15) -- (7.12,0) -- cycle    ;

%Straight Lines [id:da28659038073038334]
\draw  [dash pattern={on 4.5pt off 4.5pt}]  (188.89,60.31) -- (248.53,61.12) ;
\draw [shift={(248.53,61.12)}, rotate = 180] [fill={rgb, 255:red, 0; green, 0; blue, 0 }  ][line width=0.75]  [draw opacity=0] (10.72,-5.15) -- (0,0) -- (10.72,5.15) -- (7.12,0) -- cycle    ;

%Straight Lines [id:da7909315321853456]
\draw  [dash pattern={on 4.5pt off 4.5pt}]  (188.89,113.31) -- (248.53,114.12) ;
\draw [shift={(248.53,114.12)}, rotate = 180] [fill={rgb, 255:red, 0; green, 0; blue, 0 }  ][line width=0.75]  [draw opacity=0] (10.72,-5.15) -- (0,0) -- (10.72,5.15) -- (7.12,0) -- cycle    ;

%Straight Lines [id:da518926282878204]
\draw  [dash pattern={on 4.5pt off 4.5pt}]  (191.89,229.97) -- (251.53,230.79) ;
\draw [shift={(251.53,230.79)}, rotate = 180] [fill={rgb, 255:red, 0; green, 0; blue, 0 }  ][line width=0.75]  [draw opacity=0] (10.72,-5.15) -- (0,0) -- (10.72,5.15) -- (7.12,0) -- cycle    ;

%Straight Lines [id:da41041324240978394]
\draw  [dash pattern={on 4.5pt off 4.5pt}]  (351.26,61.49) -- (410.89,62.31) ;
\draw [shift={(410.89,62.31)}, rotate = 180] [fill={rgb, 255:red, 0; green, 0; blue, 0 }  ][line width=0.75]  [draw opacity=0] (10.72,-5.15) -- (0,0) -- (10.72,5.15) -- (7.12,0) -- cycle    ;

%Straight Lines [id:da3863807303179623]
\draw  [dash pattern={on 4.5pt off 4.5pt}]  (351.26,114.49) -- (410.89,115.31) ;
\draw [shift={(410.89,115.31)}, rotate = 180] [fill={rgb, 255:red, 0; green, 0; blue, 0 }  ][line width=0.75]  [draw opacity=0] (10.72,-5.15) -- (0,0) -- (10.72,5.15) -- (7.12,0) -- cycle    ;

%Straight Lines [id:da7751325045564748]
\draw  [dash pattern={on 4.5pt off 4.5pt}]  (354.26,231.16) -- (413.89,231.97) ;
\draw [shift={(413.89,231.97)}, rotate = 180] [fill={rgb, 255:red, 0; green, 0; blue, 0 }  ][line width=0.75]  [draw opacity=0] (10.72,-5.15) -- (0,0) -- (10.72,5.15) -- (7.12,0) -- cycle    ;

%Curve Lines [id:da416496641663032]
\draw    (410.89,62.31) .. controls (385.63,86.88) and (170.67,95.91) .. (100.93,59.85) ;
\draw [shift={(99.89,59.31)}, rotate = 388.14] [fill={rgb, 255:red, 0; green, 0; blue, 0 }  ][line width=0.75]  [draw opacity=0] (10.72,-5.15) -- (0,0) -- (10.72,5.15) -- (7.12,0) -- cycle    ;

%Curve Lines [id:da7603548964820352]
\draw    (410.95,119.06) .. controls (385.68,143.63) and (170.72,152.66) .. (100.99,116.6) ;
\draw [shift={(99.95,116.06)}, rotate = 388.14] [fill={rgb, 255:red, 0; green, 0; blue, 0 }  ][line width=0.75]  [draw opacity=0] (10.72,-5.15) -- (0,0) -- (10.72,5.15) -- (7.12,0) -- cycle    ;

%Curve Lines [id:da5823968486182294]
\draw    (413.95,235.72) .. controls (388.68,260.29) and (173.72,269.33) .. (103.99,233.27) ;
\draw [shift={(102.95,232.72)}, rotate = 388.14] [fill={rgb, 255:red, 0; green, 0; blue, 0 }  ][line width=0.75]  [draw opacity=0] (10.72,-5.15) -- (0,0) -- (10.72,5.15) -- (7.12,0) -- cycle    ;

% Text Node
\draw (101,166.33) node  [align=left] {$\displaystyle \vdots $};
% Text Node
\draw (123,43) node [scale=0.8] [align=left] {$\displaystyle {\textstyle ( 1\ 1')}$};
% Text Node
\draw (122,99) node [scale=0.8] [align=left] {$\displaystyle {\textstyle ( 1\ 2')}$};
% Text Node
\draw (123,214) node [scale=0.8] [align=left] {$\displaystyle {\textstyle ( 1\ n')}$};
% Text Node
\draw (190,167.33) node  [align=left] {$\displaystyle \vdots $};
% Text Node
\draw (212,44) node [scale=0.8] [align=left] {$\displaystyle {\textstyle ( 1\ 1')}$};
% Text Node
\draw (211,100) node [scale=0.8] [align=left] {$\displaystyle {\textstyle ( 1\ 2')}$};
% Text Node
\draw (212,215) node [scale=0.8] [align=left] {$\displaystyle {\textstyle ( 1\ n')}$};
% Text Node
\draw (301,167.33) node  [align=left] {$\displaystyle \dotsc $};
% Text Node
\draw (412,169.33) node  [align=left] {$\displaystyle \vdots $};
% Text Node
\draw (434,46) node [scale=0.8] [align=left] {$\displaystyle {\textstyle ( 1\ 1')}$};
% Text Node
\draw (433,102) node [scale=0.8] [align=left] {$\displaystyle {\textstyle ( 1\ 2')}$};
% Text Node
\draw (434,217) node [scale=0.8] [align=left] {$\displaystyle {\textstyle ( 1\ n')}$};
% Text Node
\draw (106.33,285) node  [scale = 0.8][align=left] {cell 1};
% Text Node
\draw (103.67,268) node [scale=1.7280000000000002,rotate=-269.46] [align=left] {\{};
% Text Node
\draw (190.67,285) node  [scale = 0.8][align=left] {cell 2};
% Text Node
\draw (188,268) node [scale=1.7280000000000002,rotate=-269.46] [align=left] {\{};
% Text Node
\draw (410.67,283.67) node  [scale = 0.8][align=left] {cell $\displaystyle n-1$};
% Text Node
\draw (408.67,268) node [scale=1.7280000000000002,rotate=-269.46] [align=left] {\{};

\end{tikzpicture}
\caption{Group consisting of $n-1$ cells corresponding to facility 1. \label{GroupFacility1}}
\end{figure}
\noindent We now specify the connections between the groups. Each group is connected to any other group via one of its cells. Since we have $n$ groups and each group consists of $n-1$ cells, this results in ${n \choose 2}$ connections. Connecting the cells of two groups is done by introducing a connection node and a relink node. Starting from the first group, we draw an arc from every node of one of its cells to the connection node. Successively, we draw an arc from the connection node to all the nodes of one of the cells of the second group. The same is done for the relink node, now in the reverse direction. Figure \ref{ConnectionCells} depicts an overview of the connection between the last cell of group $i$ and the first cell of group $j$. We denote the cycles between the groups by outer cycles. In Figure \ref{ConnectionCells} solid arcs are used for the outer cycles, while the inner cycles are drawn using dashed arcs. A similar connection via connection and relink nodes exists for all other pairs of groups. 

\begin{figure}[H]
\centering
\tikzset{every picture/.style={line width=0.75pt}} %set default line width to 0.75pt

\tikzset{every picture/.style={line width=0.75pt}} %set default line width to 0.75pt

\begin{tikzpicture}[scale = 0.8, x=0.75pt,y=0.75pt,yscale=-1,xscale=1]

%Shape: Rectangle [id:dp8756147911399588]
\draw   (20.5,102) -- (245.5,102) -- (245.5,203) -- (20.5,203) -- cycle ;
%Shape: Circle [id:dp9664007764972202]
\draw  [fill={rgb, 255:red, 0; green, 0; blue, 0 }  ,fill opacity=1 ] (43,117.25) .. controls (43,115.73) and (44.23,114.5) .. (45.75,114.5) .. controls (47.27,114.5) and (48.5,115.73) .. (48.5,117.25) .. controls (48.5,118.77) and (47.27,120) .. (45.75,120) .. controls (44.23,120) and (43,118.77) .. (43,117.25) -- cycle ;
%Shape: Circle [id:dp7144870531084326]
\draw  [fill={rgb, 255:red, 0; green, 0; blue, 0 }  ,fill opacity=1 ] (43,137.75) .. controls (43,136.23) and (44.23,135) .. (45.75,135) .. controls (47.27,135) and (48.5,136.23) .. (48.5,137.75) .. controls (48.5,139.27) and (47.27,140.5) .. (45.75,140.5) .. controls (44.23,140.5) and (43,139.27) .. (43,137.75) -- cycle ;
%Shape: Circle [id:dp39850083284398474]
\draw  [fill={rgb, 255:red, 0; green, 0; blue, 0 }  ,fill opacity=1 ] (43.5,186.25) .. controls (43.5,184.73) and (44.73,183.5) .. (46.25,183.5) .. controls (47.77,183.5) and (49,184.73) .. (49,186.25) .. controls (49,187.77) and (47.77,189) .. (46.25,189) .. controls (44.73,189) and (43.5,187.77) .. (43.5,186.25) -- cycle ;
%Shape: Circle [id:dp44071027061035406]
\draw  [fill={rgb, 255:red, 0; green, 0; blue, 0 }  ,fill opacity=1 ] (89,117.25) .. controls (89,115.73) and (90.23,114.5) .. (91.75,114.5) .. controls (93.27,114.5) and (94.5,115.73) .. (94.5,117.25) .. controls (94.5,118.77) and (93.27,120) .. (91.75,120) .. controls (90.23,120) and (89,118.77) .. (89,117.25) -- cycle ;
%Shape: Circle [id:dp5952004456818576]
\draw  [fill={rgb, 255:red, 0; green, 0; blue, 0 }  ,fill opacity=1 ] (89,137.75) .. controls (89,136.23) and (90.23,135) .. (91.75,135) .. controls (93.27,135) and (94.5,136.23) .. (94.5,137.75) .. controls (94.5,139.27) and (93.27,140.5) .. (91.75,140.5) .. controls (90.23,140.5) and (89,139.27) .. (89,137.75) -- cycle ;
%Shape: Circle [id:dp4906045892664277]
\draw  [fill={rgb, 255:red, 0; green, 0; blue, 0 }  ,fill opacity=1 ] (89.5,186.25) .. controls (89.5,184.73) and (90.73,183.5) .. (92.25,183.5) .. controls (93.77,183.5) and (95,184.73) .. (95,186.25) .. controls (95,187.77) and (93.77,189) .. (92.25,189) .. controls (90.73,189) and (89.5,187.77) .. (89.5,186.25) -- cycle ;
%Shape: Circle [id:dp5375982180468577]
\draw  [fill={rgb, 255:red, 0; green, 0; blue, 0 }  ,fill opacity=1 ] (211,118.25) .. controls (211,116.73) and (212.23,115.5) .. (213.75,115.5) .. controls (215.27,115.5) and (216.5,116.73) .. (216.5,118.25) .. controls (216.5,119.77) and (215.27,121) .. (213.75,121) .. controls (212.23,121) and (211,119.77) .. (211,118.25) -- cycle ;
%Shape: Circle [id:dp7933630487137915]
\draw  [fill={rgb, 255:red, 0; green, 0; blue, 0 }  ,fill opacity=1 ] (211,138.75) .. controls (211,137.23) and (212.23,136) .. (213.75,136) .. controls (215.27,136) and (216.5,137.23) .. (216.5,138.75) .. controls (216.5,140.27) and (215.27,141.5) .. (213.75,141.5) .. controls (212.23,141.5) and (211,140.27) .. (211,138.75) -- cycle ;
%Shape: Circle [id:dp29542613453113487]
\draw  [fill={rgb, 255:red, 0; green, 0; blue, 0 }  ,fill opacity=1 ] (211.5,187.25) .. controls (211.5,185.73) and (212.73,184.5) .. (214.25,184.5) .. controls (215.77,184.5) and (217,185.73) .. (217,187.25) .. controls (217,188.77) and (215.77,190) .. (214.25,190) .. controls (212.73,190) and (211.5,188.77) .. (211.5,187.25) -- cycle ;
%Straight Lines [id:da42439420769634384]
\draw  [dash pattern={on 4.5pt off 4.5pt}]  (45.75,117.25) -- (91.75,117.25) ;

%Straight Lines [id:da7804930081083206]
\draw  [dash pattern={on 4.5pt off 4.5pt}]  (45.75,137.75) -- (91.75,137.75) ;

%Straight Lines [id:da141798021370231]
\draw  [dash pattern={on 4.5pt off 4.5pt}]  (49,186.25) -- (95,186.25) ;

%Straight Lines [id:da8187040129688152]
\draw  [dash pattern={on 4.5pt off 4.5pt}]  (91.75,117.25) -- (120.88,117.5) ;

%Straight Lines [id:da7674228415963762]
\draw  [dash pattern={on 4.5pt off 4.5pt}]  (91.75,137.75) -- (120.88,138) ;

%Straight Lines [id:da4923823420929261]
\draw  [dash pattern={on 4.5pt off 4.5pt}]  (92.25,186.25) -- (121.38,186.5) ;

%Straight Lines [id:da20471714846637967]
\draw  [dash pattern={on 4.5pt off 4.5pt}]  (184.5,118) -- (213.75,118.25) ;

%Straight Lines [id:da9878227064272223]
\draw  [dash pattern={on 4.5pt off 4.5pt}]  (184.63,138.5) -- (213.75,138.75) ;

%Straight Lines [id:da37315676390476527]
\draw  [dash pattern={on 4.5pt off 4.5pt}]  (185.13,187) -- (214.25,187.25) ;

%Curve Lines [id:da5460750938028824]
\draw  [dash pattern={on 4.5pt off 4.5pt}]  (45.75,117.25) .. controls (87.5,126) and (161.5,132) .. (213.75,118.25) ;

%Curve Lines [id:da7736997617596097]
\draw  [dash pattern={on 4.5pt off 4.5pt}]  (43,137.75) .. controls (84.75,146.5) and (158.75,152.5) .. (211,138.75) ;

%Curve Lines [id:da1536372395635357]
\draw  [dash pattern={on 4.5pt off 4.5pt}]  (46.25,186.25) .. controls (88,195) and (162,201) .. (214.25,187.25) ;

%Shape: Rectangle [id:dp34906731973144445]
\draw   (412,102) -- (641.5,102) -- (641.5,203) -- (412,203) -- cycle ;
%Shape: Circle [id:dp14228461516819313]
\draw  [fill={rgb, 255:red, 0; green, 0; blue, 0 }  ,fill opacity=1 ] (445,117.25) .. controls (445,115.73) and (446.23,114.5) .. (447.75,114.5) .. controls (449.27,114.5) and (450.5,115.73) .. (450.5,117.25) .. controls (450.5,118.77) and (449.27,120) .. (447.75,120) .. controls (446.23,120) and (445,118.77) .. (445,117.25) -- cycle ;
%Shape: Circle [id:dp4133269884610571]
\draw  [fill={rgb, 255:red, 0; green, 0; blue, 0 }  ,fill opacity=1 ] (445,137.75) .. controls (445,136.23) and (446.23,135) .. (447.75,135) .. controls (449.27,135) and (450.5,136.23) .. (450.5,137.75) .. controls (450.5,139.27) and (449.27,140.5) .. (447.75,140.5) .. controls (446.23,140.5) and (445,139.27) .. (445,137.75) -- cycle ;
%Shape: Circle [id:dp8933034730373495]
\draw  [fill={rgb, 255:red, 0; green, 0; blue, 0 }  ,fill opacity=1 ] (445.5,186.25) .. controls (445.5,184.73) and (446.73,183.5) .. (448.25,183.5) .. controls (449.77,183.5) and (451,184.73) .. (451,186.25) .. controls (451,187.77) and (449.77,189) .. (448.25,189) .. controls (446.73,189) and (445.5,187.77) .. (445.5,186.25) -- cycle ;
%Shape: Circle [id:dp8354667837286716]
\draw  [fill={rgb, 255:red, 0; green, 0; blue, 0 }  ,fill opacity=1 ] (491,117.25) .. controls (491,115.73) and (492.23,114.5) .. (493.75,114.5) .. controls (495.27,114.5) and (496.5,115.73) .. (496.5,117.25) .. controls (496.5,118.77) and (495.27,120) .. (493.75,120) .. controls (492.23,120) and (491,118.77) .. (491,117.25) -- cycle ;
%Shape: Circle [id:dp3763052602794077]
\draw  [fill={rgb, 255:red, 0; green, 0; blue, 0 }  ,fill opacity=1 ] (491,137.75) .. controls (491,136.23) and (492.23,135) .. (493.75,135) .. controls (495.27,135) and (496.5,136.23) .. (496.5,137.75) .. controls (496.5,139.27) and (495.27,140.5) .. (493.75,140.5) .. controls (492.23,140.5) and (491,139.27) .. (491,137.75) -- cycle ;
%Shape: Circle [id:dp0785738982551547]
\draw  [fill={rgb, 255:red, 0; green, 0; blue, 0 }  ,fill opacity=1 ] (491.5,186.25) .. controls (491.5,184.73) and (492.73,183.5) .. (494.25,183.5) .. controls (495.77,183.5) and (497,184.73) .. (497,186.25) .. controls (497,187.77) and (495.77,189) .. (494.25,189) .. controls (492.73,189) and (491.5,187.77) .. (491.5,186.25) -- cycle ;
%Shape: Circle [id:dp004007751730548792]
\draw  [fill={rgb, 255:red, 0; green, 0; blue, 0 }  ,fill opacity=1 ] (613,118.25) .. controls (613,116.73) and (614.23,115.5) .. (615.75,115.5) .. controls (617.27,115.5) and (618.5,116.73) .. (618.5,118.25) .. controls (618.5,119.77) and (617.27,121) .. (615.75,121) .. controls (614.23,121) and (613,119.77) .. (613,118.25) -- cycle ;
%Shape: Circle [id:dp2205616447272225]
\draw  [fill={rgb, 255:red, 0; green, 0; blue, 0 }  ,fill opacity=1 ] (613,138.75) .. controls (613,137.23) and (614.23,136) .. (615.75,136) .. controls (617.27,136) and (618.5,137.23) .. (618.5,138.75) .. controls (618.5,140.27) and (617.27,141.5) .. (615.75,141.5) .. controls (614.23,141.5) and (613,140.27) .. (613,138.75) -- cycle ;
%Shape: Circle [id:dp4945043660260817]
\draw  [fill={rgb, 255:red, 0; green, 0; blue, 0 }  ,fill opacity=1 ] (613.5,187.25) .. controls (613.5,185.73) and (614.73,184.5) .. (616.25,184.5) .. controls (617.77,184.5) and (619,185.73) .. (619,187.25) .. controls (619,188.77) and (617.77,190) .. (616.25,190) .. controls (614.73,190) and (613.5,188.77) .. (613.5,187.25) -- cycle ;
%Straight Lines [id:da2906560956203039]
\draw  [dash pattern={on 4.5pt off 4.5pt}]  (447.75,117.25) -- (493.75,117.25) ;

%Straight Lines [id:da6830542805311792]
\draw  [dash pattern={on 4.5pt off 4.5pt}]  (447.75,137.75) -- (493.75,137.75) ;

%Straight Lines [id:da7260209756935145]
\draw  [dash pattern={on 4.5pt off 4.5pt}]  (451,186.25) -- (497,186.25) ;

%Straight Lines [id:da15185596881227248]
\draw  [dash pattern={on 4.5pt off 4.5pt}]  (493.75,117.25) -- (522.88,117.5) ;

%Straight Lines [id:da709390064072658]
\draw  [dash pattern={on 4.5pt off 4.5pt}]  (493.75,137.75) -- (522.88,138) ;

%Straight Lines [id:da34773800265893007]
\draw  [dash pattern={on 4.5pt off 4.5pt}]  (494.25,186.25) -- (523.38,186.5) ;

%Straight Lines [id:da034655200434956956]
\draw  [dash pattern={on 4.5pt off 4.5pt}]  (586.5,118) -- (615.75,118.25) ;

%Straight Lines [id:da7856071706771068]
\draw  [dash pattern={on 4.5pt off 4.5pt}]  (586.63,138.5) -- (615.75,138.75) ;

%Straight Lines [id:da6702114838703819]
\draw  [dash pattern={on 4.5pt off 4.5pt}]  (587.13,187) -- (616.25,187.25) ;

%Curve Lines [id:da22174183758639443]
\draw  [dash pattern={on 4.5pt off 4.5pt}]  (447.75,117.25) .. controls (489.5,126) and (563.5,132) .. (615.75,118.25) ;

%Curve Lines [id:da5680041323614216]
\draw  [dash pattern={on 4.5pt off 4.5pt}]  (445,137.75) .. controls (486.75,146.5) and (560.75,152.5) .. (613,138.75) ;

%Curve Lines [id:da6858627825136685]
\draw  [dash pattern={on 4.5pt off 4.5pt}]  (448.25,186.25) .. controls (490,195) and (564,201) .. (616.25,187.25) ;

%Shape: Circle [id:dp6354079443867116]
\draw  [fill={rgb, 255:red, 0; green, 0; blue, 0 }  ,fill opacity=1 ] (324.5,112.25) .. controls (324.5,110.73) and (325.73,109.5) .. (327.25,109.5) .. controls (328.77,109.5) and (330,110.73) .. (330,112.25) .. controls (330,113.77) and (328.77,115) .. (327.25,115) .. controls (325.73,115) and (324.5,113.77) .. (324.5,112.25) -- cycle ;
%Shape: Circle [id:dp9480377429428992]
\draw  [color={rgb, 255:red, 0; green, 0; blue, 0 }  ,draw opacity=1 ][fill={rgb, 255:red, 0; green, 0; blue, 0 }  ,fill opacity=1 ] (324.5,188.25) .. controls (324.5,186.73) and (325.73,185.5) .. (327.25,185.5) .. controls (328.77,185.5) and (330,186.73) .. (330,188.25) .. controls (330,189.77) and (328.77,191) .. (327.25,191) .. controls (325.73,191) and (324.5,189.77) .. (324.5,188.25) -- cycle ;
%Straight Lines [id:da3769664738257297]
\draw    (216.5,118.25) -- (322.5,112.36) ;
\draw [shift={(324.5,112.25)}, rotate = 536.8199999999999] [fill={rgb, 255:red, 0; green, 0; blue, 0 }  ][line width=0.75]  [draw opacity=0] (8.93,-4.29) -- (0,0) -- (8.93,4.29) -- cycle    ;

%Straight Lines [id:da34301567862633187]
\draw    (216.5,138.75) -- (322.56,112.73) ;
\draw [shift={(324.5,112.25)}, rotate = 526.21] [fill={rgb, 255:red, 0; green, 0; blue, 0 }  ][line width=0.75]  [draw opacity=0] (8.93,-4.29) -- (0,0) -- (8.93,4.29) -- cycle    ;

%Straight Lines [id:da511002968884283]
\draw    (214.25,187.25) -- (322.85,113.37) ;
\draw [shift={(324.5,112.25)}, rotate = 505.77] [fill={rgb, 255:red, 0; green, 0; blue, 0 }  ][line width=0.75]  [draw opacity=0] (8.93,-4.29) -- (0,0) -- (8.93,4.29) -- cycle    ;

%Straight Lines [id:da6386075578834223]
\draw    (327.25,112.25) -- (443,117.17) ;
\draw [shift={(445,117.25)}, rotate = 182.43] [fill={rgb, 255:red, 0; green, 0; blue, 0 }  ][line width=0.75]  [draw opacity=0] (8.93,-4.29) -- (0,0) -- (8.93,4.29) -- cycle    ;

%Straight Lines [id:da6931571594796511]
\draw    (327.25,112.25) -- (443.05,137.33) ;
\draw [shift={(445,137.75)}, rotate = 192.22] [fill={rgb, 255:red, 0; green, 0; blue, 0 }  ][line width=0.75]  [draw opacity=0] (8.93,-4.29) -- (0,0) -- (8.93,4.29) -- cycle    ;

%Straight Lines [id:da5822599990246737]
\draw    (330,112.25) -- (443.82,185.17) ;
\draw [shift={(445.5,186.25)}, rotate = 212.65] [fill={rgb, 255:red, 0; green, 0; blue, 0 }  ][line width=0.75]  [draw opacity=0] (8.93,-4.29) -- (0,0) -- (8.93,4.29) -- cycle    ;

%Straight Lines [id:da19217305832692122]
\draw [color={rgb, 255:red, 0; green, 0; blue, 0 }  ,draw opacity=1 ][fill={rgb, 255:red, 208; green, 2; blue, 27 }  ,fill opacity=1 ]   (445,117.25) -- (331.7,187.2) ;
\draw [shift={(330,188.25)}, rotate = 328.31] [fill={rgb, 255:red, 0; green, 0; blue, 0 }  ,fill opacity=1 ][line width=0.75]  [draw opacity=0] (8.93,-4.29) -- (0,0) -- (8.93,4.29) -- cycle    ;

%Straight Lines [id:da9272787373883171]
\draw [color={rgb, 255:red, 0; green, 0; blue, 0 }  ,draw opacity=1 ]   (445,137.75) -- (331.83,187.45) ;
\draw [shift={(330,188.25)}, rotate = 336.28999999999996] [fill={rgb, 255:red, 0; green, 0; blue, 0 }  ,fill opacity=1 ][line width=0.75]  [draw opacity=0] (8.93,-4.29) -- (0,0) -- (8.93,4.29) -- cycle    ;

%Straight Lines [id:da8245692252188541]
\draw [color={rgb, 255:red, 0; green, 0; blue, 0 }  ,draw opacity=1 ][fill={rgb, 255:red, 208; green, 2; blue, 27 }  ,fill opacity=1 ]   (445.5,186.25) -- (332,188.22) ;
\draw [shift={(330,188.25)}, rotate = 359.01] [fill={rgb, 255:red, 0; green, 0; blue, 0 }  ,fill opacity=1 ][line width=0.75]  [draw opacity=0] (8.93,-4.29) -- (0,0) -- (8.93,4.29) -- cycle    ;

%Straight Lines [id:da6842551199667009]
\draw [color={rgb, 255:red, 0; green, 0; blue, 0 }  ,draw opacity=1 ]   (324.5,188.25) -- (219,187.27) ;
\draw [shift={(217,187.25)}, rotate = 360.53] [fill={rgb, 255:red, 0; green, 0; blue, 0 }  ,fill opacity=1 ][line width=0.75]  [draw opacity=0] (8.93,-4.29) -- (0,0) -- (8.93,4.29) -- cycle    ;

%Straight Lines [id:da20056369386609285]
\draw [color={rgb, 255:red, 0; green, 0; blue, 0 }  ,draw opacity=1 ]   (324.5,188.25) -- (218.32,139.58) ;
\draw [shift={(216.5,138.75)}, rotate = 384.62] [fill={rgb, 255:red, 0; green, 0; blue, 0 }  ,fill opacity=1 ][line width=0.75]  [draw opacity=0] (8.93,-4.29) -- (0,0) -- (8.93,4.29) -- cycle    ;

%Straight Lines [id:da2115934327997595]
\draw [color={rgb, 255:red, 0; green, 0; blue, 0 }  ,draw opacity=1 ]   (324.5,188.25) -- (218.18,119.34) ;
\draw [shift={(216.5,118.25)}, rotate = 392.95] [fill={rgb, 255:red, 0; green, 0; blue, 0 }  ,fill opacity=1 ][line width=0.75]  [draw opacity=0] (8.93,-4.29) -- (0,0) -- (8.93,4.29) -- cycle    ;

% Text Node
\draw (45.96,158.73) node  [align=left] {$\displaystyle \vdots $};
% Text Node
\draw (91.96,158.73) node  [align=left] {$\displaystyle \vdots $};
% Text Node
\draw (213.96,159.73) node  [align=left] {$\displaystyle \vdots $};
% Text Node
\draw (151.76,160.93) node [rotate=-89.44] [align=left] {$\displaystyle \vdots $};
% Text Node
\draw (447.96,158.73) node  [align=left] {$\displaystyle \vdots $};
% Text Node
\draw (493.96,158.73) node  [align=left] {$\displaystyle \vdots $};
% Text Node
\draw (615.96,159.73) node  [align=left] {$\displaystyle \vdots $};
% Text Node
\draw (553.76,160.93) node [rotate=-89.44] [align=left] {$\displaystyle \vdots $};
% Text Node
\draw (130,82) node [scale=0.9] [align=left] {group $\displaystyle i$};
% Text Node
\draw (533,83) node [scale=0.9] [align=left] {group $\displaystyle j$};
% Text Node
\draw (329,84) node [scale=0.9] [align=left] {connection node};
% Text Node
\draw (330,212) node [scale=0.9] [align=left] {relink node};

\end{tikzpicture}
\caption{Connection between two cells of group $i$ and $j$. \label{ConnectionCells}}

\end{figure}
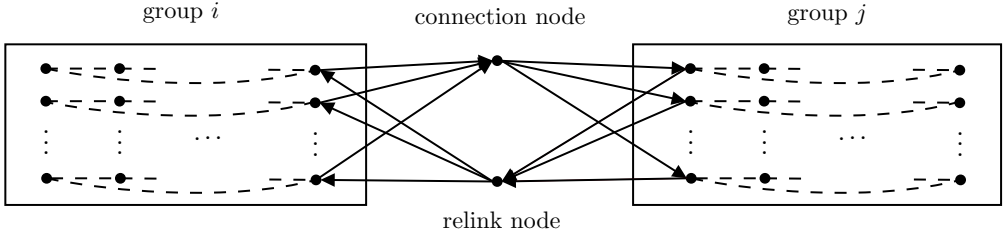

{ \noindent  Observe that any arc in $G$ either belongs to an inner or an outer cycle. The quadratic cost of a pair of successive arcs $(e,f)$ where $e$ belongs to an inner cycle and $f$ to an outer cycle or vice versa, is set to $\infty$. It remains to specify the interaction cost between successive arcs on an outer cycle. We only specify the quadratic cost between the arcs entering and leaving the connection node, other costs are set to zero.

Let $i$ and $j$ be two distinct groups associated with facility $i$ and $j$, respectively. Let a node in group $i$ be given by $(ik')$ with $k' \in P$. Similarly, a node in group $j$ is given by $(jl')$ with $l' \in P$. Let $e_{ik'}$ denote the arc between $(ik')$ and the connection node and let $f_{jl'}$ denote the arc between the connection node and $(jl')$. Then the quadratic cost between $e_{ik'}$ and $f_{jl'}$ is defined as follows:
\begin{align*}
Q_{e_{ik'}, f_{jl'}} := \begin{cases} d_{k'l'}w_{ij} + d_{l'k'}w_{ji} & \text{if $k' \neq l'$} \\
\infty & \text{otherwise.}
\end{cases}
\end{align*}
We repeat this construction for any two connected cells. Figure \ref{OverviewG} gives a simplified overview of $G$ for $n = 4$. The circles in the center denote the connections between the cells, where the connection and relink nodes are drawn using the symbols `$\bullet$' and `$\ast$', respectively. The graph $G$ has $n^2(n-1) + 2{n \choose 2} = \mathcal{O}(n^3)$ nodes and $n^2(n-1) + 4{n \choose 2} = \mathcal{O}(n^3)$ arcs.  }

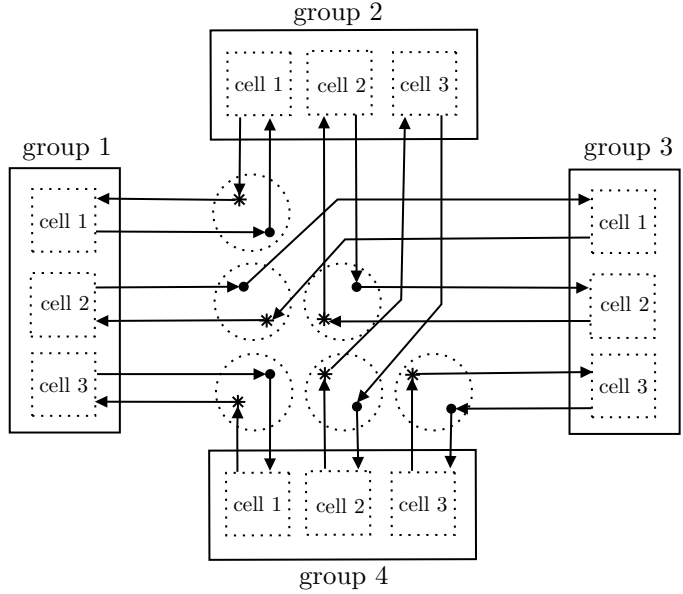
\begin{figure}[h]
\centering

\tikzset{every picture/.style={line width=0.75pt}} %set default line width to 0.75pt

\begin{tikzpicture}[scale = 0.7, x=0.75pt,y=0.75pt,yscale=-1,xscale=1]
%uncomment if require: \path (0,500); %set diagram left start at 0, and has height of 500

%Shape: Rectangle [id:dp9490693529128986]
\draw   (74,164.5) -- (152.5,164.5) -- (152.5,354.5) -- (74,354.5) -- cycle ;
%Shape: Square [id:dp7066518635341619]
\draw  [dash pattern={on 0.84pt off 2.51pt}] (90,179.5) -- (135,179.5) -- (135,224.5) -- (90,224.5) -- cycle ;
%Shape: Square [id:dp5871963278683265]
\draw  [dash pattern={on 0.84pt off 2.51pt}] (89,239.83) -- (134.67,239.83) -- (134.67,285.5) -- (89,285.5) -- cycle ;
%Shape: Square [id:dp3073264863810361]
\draw  [dash pattern={on 0.84pt off 2.51pt}] (90,297.5) -- (135,297.5) -- (135,342.5) -- (90,342.5) -- cycle ;
%Shape: Rectangle [id:dp4683007817301237]
\draw   (473,165.5) -- (551.5,165.5) -- (551.5,355.5) -- (473,355.5) -- cycle ;
%Shape: Square [id:dp5842885583943938]
\draw  [dash pattern={on 0.84pt off 2.51pt}] (489,180.5) -- (534,180.5) -- (534,225.5) -- (489,225.5) -- cycle ;
%Shape: Square [id:dp7528768939912094]
\draw  [dash pattern={on 0.84pt off 2.51pt}] (488,240.83) -- (533.67,240.83) -- (533.67,286.5) -- (488,286.5) -- cycle ;
%Shape: Square [id:dp7695029460908998]
\draw  [dash pattern={on 0.84pt off 2.51pt}] (489,298.5) -- (534,298.5) -- (534,343.5) -- (489,343.5) -- cycle ;
%Shape: Rectangle [id:dp7683947984899713]
\draw   (407.15,65) -- (407.35,143.5) -- (217.35,144) -- (217.15,65.5) -- cycle ;
%Shape: Square [id:dp8862342697896071]
\draw  [dash pattern={on 0.84pt off 2.51pt}] (392.19,81.04) -- (392.31,126.04) -- (347.31,126.16) -- (347.19,81.16) -- cycle ;
%Shape: Square [id:dp414051808262043]
\draw  [dash pattern={on 0.84pt off 2.51pt}] (331.85,80.2) -- (331.97,125.87) -- (286.31,125.99) -- (286.19,80.32) -- cycle ;
%Shape: Square [id:dp5190748050756668]
\draw  [dash pattern={on 0.84pt off 2.51pt}] (274.19,81.35) -- (274.31,126.35) -- (229.31,126.47) -- (229.19,81.47) -- cycle ;
%Shape: Rectangle [id:dp9670276765574222]
\draw   (406.15,367) -- (406.35,445.5) -- (216.35,446) -- (216.15,367.5) -- cycle ;
%Shape: Square [id:dp9481038459850923]
\draw  [dash pattern={on 0.84pt off 2.51pt}] (391.19,383.04) -- (391.31,428.04) -- (346.31,428.16) -- (346.19,383.16) -- cycle ;
%Shape: Square [id:dp7243591231933018]
\draw  [dash pattern={on 0.84pt off 2.51pt}] (330.85,382.2) -- (330.97,427.87) -- (285.31,427.99) -- (285.19,382.32) -- cycle ;
%Shape: Square [id:dp10818067957172484]
\draw  [dash pattern={on 0.84pt off 2.51pt}] (273.19,383.35) -- (273.31,428.35) -- (228.31,428.47) -- (228.19,383.47) -- cycle ;
%Straight Lines [id:da9452898234622631]
\draw    (135.5,210) -- (254.48,210.57) ;
\draw [shift={(256.47,210.58)}, rotate = 180.27] [fill={rgb, 255:red, 0; green, 0; blue, 0 }  ][line width=0.75]  [draw opacity=0] (8.93,-4.29) -- (0,0) -- (8.93,4.29) -- cycle    ;

%Straight Lines [id:da38459452195144794]
\draw    (135.25,250.25) -- (235.9,249.7) ;
\draw [shift={(237.9,249.69)}, rotate = 539.69] [fill={rgb, 255:red, 0; green, 0; blue, 0 }  ][line width=0.75]  [draw opacity=0] (8.93,-4.29) -- (0,0) -- (8.93,4.29) -- cycle    ;

%Straight Lines [id:da08552331031904714]
\draw    (136.25,312.5) -- (254.75,312.5) ;
\draw [shift={(256.75,312.5)}, rotate = 180] [fill={rgb, 255:red, 0; green, 0; blue, 0 }  ][line width=0.75]  [draw opacity=0] (8.93,-4.29) -- (0,0) -- (8.93,4.29) -- cycle    ;

%Straight Lines [id:da5720752356058676]
\draw    (259.35,207.7) -- (259.5,130.5) ;
\draw [shift={(259.5,128.5)}, rotate = 450.11] [fill={rgb, 255:red, 0; green, 0; blue, 0 }  ][line width=0.75]  [draw opacity=0] (8.93,-4.29) -- (0,0) -- (8.93,4.29) -- cycle    ;

%Straight Lines [id:da39511226242949893]
\draw    (240.77,249.69) -- (307.69,187.02) -- (486.35,187.02) ;
\draw [shift={(488.35,187.02)}, rotate = 180] [fill={rgb, 255:red, 0; green, 0; blue, 0 }  ][line width=0.75]  [draw opacity=0] (8.93,-4.29) -- (0,0) -- (8.93,4.29) -- cycle    ;

%Straight Lines [id:da9150945083644011]
\draw    (259.63,315.38) -- (259.5,380) ;
\draw [shift={(259.5,382)}, rotate = 270.11] [fill={rgb, 255:red, 0; green, 0; blue, 0 }  ][line width=0.75]  [draw opacity=0] (8.93,-4.29) -- (0,0) -- (8.93,4.29) -- cycle    ;

%Straight Lines [id:da7766360387693743]
\draw [color={rgb, 255:red, 0; green, 0; blue, 0 }  ,draw opacity=1 ]   (298,129) -- (298.3,271.8) ;

\draw [shift={(298,127)}, rotate = 89.88] [fill={rgb, 255:red, 0; green, 0; blue, 0 }  ,fill opacity=1 ][line width=0.75]  [draw opacity=0] (8.93,-4.29) -- (0,0) -- (8.93,4.29) -- cycle    ;
%Straight Lines [id:da12603057027741316]
\draw [color={rgb, 255:red, 0; green, 0; blue, 0 }  ,draw opacity=1 ]   (303.17,274.67) -- (488,274.33) ;

\draw [shift={(301.18,274.68)}, rotate = 359.9] [fill={rgb, 255:red, 0; green, 0; blue, 0 }  ,fill opacity=1 ][line width=0.75]  [draw opacity=0] (8.93,-4.29) -- (0,0) -- (8.93,4.29) -- cycle    ;
%Straight Lines [id:da7887447101407727]
\draw    (381.75,126.5) -- (381.69,262.35) -- (322.8,331.48) ;
\draw [shift={(321.5,333)}, rotate = 310.43] [fill={rgb, 255:red, 0; green, 0; blue, 0 }  ][line width=0.75]  [draw opacity=0] (8.93,-4.29) -- (0,0) -- (8.93,4.29) -- cycle    ;

%Straight Lines [id:da6187659762529862]
\draw    (321.5,338.75) -- (322.45,379) ;
\draw [shift={(322.5,381)}, rotate = 268.64] [fill={rgb, 255:red, 0; green, 0; blue, 0 }  ][line width=0.75]  [draw opacity=0] (8.93,-4.29) -- (0,0) -- (8.93,4.29) -- cycle    ;

%Straight Lines [id:da9512995630282632]
\draw    (488.5,337) -- (393.62,337.37) ;
\draw [shift={(391.63,337.38)}, rotate = 359.78] [fill={rgb, 255:red, 0; green, 0; blue, 0 }  ][line width=0.75]  [draw opacity=0] (8.93,-4.29) -- (0,0) -- (8.93,4.29) -- cycle    ;

%Straight Lines [id:da17703583870958606]
\draw    (388.75,340.25) -- (388.04,378.25) ;
\draw [shift={(388,380.25)}, rotate = 271.07] [fill={rgb, 255:red, 0; green, 0; blue, 0 }  ][line width=0.75]  [draw opacity=0] (8.93,-4.29) -- (0,0) -- (8.93,4.29) -- cycle    ;

%Straight Lines [id:da21991393876890242]
\draw [color={rgb, 255:red, 0; green, 0; blue, 0 }  ,draw opacity=1 ]   (237.77,182.5) -- (238.25,128) ;

\draw [shift={(237.75,184.5)}, rotate = 270.51] [fill={rgb, 255:red, 0; green, 0; blue, 0 }  ,fill opacity=1 ][line width=0.75]  [draw opacity=0] (8.93,-4.29) -- (0,0) -- (8.93,4.29) -- cycle    ;
%Straight Lines [id:da9192035298254426]
\draw [color={rgb, 255:red, 0; green, 0; blue, 0 }  ,draw opacity=1 ]   (138,186.35) -- (234.88,187.38) ;

\draw [shift={(136,186.33)}, rotate = 0.6] [fill={rgb, 255:red, 0; green, 0; blue, 0 }  ,fill opacity=1 ][line width=0.75]  [draw opacity=0] (8.93,-4.29) -- (0,0) -- (8.93,4.29) -- cycle    ;
%Straight Lines [id:da8049048423870555]
\draw [color={rgb, 255:red, 0; green, 0; blue, 0 }  ,draw opacity=1 ]   (136.67,332.34) -- (233.75,332.5) ;

\draw [shift={(134.67,332.33)}, rotate = 0.1] [fill={rgb, 255:red, 0; green, 0; blue, 0 }  ,fill opacity=1 ][line width=0.75]  [draw opacity=0] (8.93,-4.29) -- (0,0) -- (8.93,4.29) -- cycle    ;
%Straight Lines [id:da8745087421117708]
\draw [color={rgb, 255:red, 0; green, 0; blue, 0 }  ,draw opacity=1 ]   (236.31,382.75) -- (236.19,352.25) -- (236.57,337.37) ;
\draw [shift={(236.63,335.38)}, rotate = 451.46] [fill={rgb, 255:red, 0; green, 0; blue, 0 }  ,fill opacity=1 ][line width=0.75]  [draw opacity=0] (8.93,-4.29) -- (0,0) -- (8.93,4.29) -- cycle    ;

%Straight Lines [id:da9791544766989726]
\draw [color={rgb, 255:red, 0; green, 0; blue, 0 }  ,draw opacity=1 ]   (363.63,312.38) -- (487.75,311.02) ;
\draw [shift={(489.75,311)}, rotate = 539.38] [fill={rgb, 255:red, 0; green, 0; blue, 0 }  ,fill opacity=1 ][line width=0.75]  [draw opacity=0] (8.93,-4.29) -- (0,0) -- (8.93,4.29) -- cycle    ;

%Straight Lines [id:da16920279297616303]
\draw [color={rgb, 255:red, 0; green, 0; blue, 0 }  ,draw opacity=1 ]   (360.75,382) -- (360.75,317.25) ;
\draw [shift={(360.75,315.25)}, rotate = 450] [fill={rgb, 255:red, 0; green, 0; blue, 0 }  ,fill opacity=1 ][line width=0.75]  [draw opacity=0] (8.93,-4.29) -- (0,0) -- (8.93,4.29) -- cycle    ;

%Straight Lines [id:da4532196641467767]
\draw [color={rgb, 255:red, 0; green, 0; blue, 0 }  ,draw opacity=1 ]   (137.33,274.32) -- (255.27,273.69) ;

\draw [shift={(135.33,274.33)}, rotate = 359.69] [fill={rgb, 255:red, 0; green, 0; blue, 0 }  ,fill opacity=1 ][line width=0.75]  [draw opacity=0] (8.93,-4.29) -- (0,0) -- (8.93,4.29) -- cycle    ;
%Straight Lines [id:da4107561302335634]
\draw [color={rgb, 255:red, 0; green, 0; blue, 0 }  ,draw opacity=1 ]   (262.36,272.2) -- (313.02,215.69) -- (487.69,214.35) ;

\draw [shift={(261.02,273.69)}, rotate = 311.88] [fill={rgb, 255:red, 0; green, 0; blue, 0 }  ,fill opacity=1 ][line width=0.75]  [draw opacity=0] (8.93,-4.29) -- (0,0) -- (8.93,4.29) -- cycle    ;
%Shape: Circle [id:dp288797257999422]
\draw  [fill={rgb, 255:red, 0; green, 0; blue, 0 }  ,fill opacity=1 ] (256.47,210.58) .. controls (256.47,208.99) and (257.76,207.7) .. (259.35,207.7) .. controls (260.94,207.7) and (262.22,208.99) .. (262.22,210.58) .. controls (262.22,212.16) and (260.94,213.45) .. (259.35,213.45) .. controls (257.76,213.45) and (256.47,212.16) .. (256.47,210.58) -- cycle ;
%Shape: Circle [id:dp8890991771544803]
\draw  [fill={rgb, 255:red, 0; green, 0; blue, 0 }  ,fill opacity=1 ] (237.9,249.69) .. controls (237.9,248.1) and (239.18,246.81) .. (240.77,246.81) .. controls (242.36,246.81) and (243.65,248.1) .. (243.65,249.69) .. controls (243.65,251.28) and (242.36,252.56) .. (240.77,252.56) .. controls (239.18,252.56) and (237.9,251.28) .. (237.9,249.69) -- cycle ;
%Shape: Circle [id:dp3789869681887996]
\draw  [fill={rgb, 255:red, 0; green, 0; blue, 0 }  ,fill opacity=1 ] (256.75,312.5) .. controls (256.75,310.91) and (258.04,309.63) .. (259.63,309.63) .. controls (261.21,309.63) and (262.5,310.91) .. (262.5,312.5) .. controls (262.5,314.09) and (261.21,315.38) .. (259.63,315.38) .. controls (258.04,315.38) and (256.75,314.09) .. (256.75,312.5) -- cycle ;
%Shape: Circle [id:dp1189925580910951]
\draw  [fill={rgb, 255:red, 0; green, 0; blue, 0 }  ,fill opacity=1 ] (318.38,249.88) .. controls (318.38,248.29) and (319.66,247) .. (321.25,247) .. controls (322.84,247) and (324.13,248.29) .. (324.13,249.88) .. controls (324.13,251.46) and (322.84,252.75) .. (321.25,252.75) .. controls (319.66,252.75) and (318.38,251.46) .. (318.38,249.88) -- cycle ;
%Straight Lines [id:da5101991241688519]
\draw    (320.75,126) -- (321.24,245) ;
\draw [shift={(321.25,247)}, rotate = 269.76] [fill={rgb, 255:red, 0; green, 0; blue, 0 }  ][line width=0.75]  [draw opacity=0] (8.93,-4.29) -- (0,0) -- (8.93,4.29) -- cycle    ;

%Straight Lines [id:da05897598385690861]
\draw    (324.13,249.88) -- (484.75,250.49) ;
\draw [shift={(486.75,250.5)}, rotate = 180.22] [fill={rgb, 255:red, 0; green, 0; blue, 0 }  ][line width=0.75]  [draw opacity=0] (8.93,-4.29) -- (0,0) -- (8.93,4.29) -- cycle    ;

%Shape: Circle [id:dp6650372975987835]
\draw  [fill={rgb, 255:red, 0; green, 0; blue, 0 }  ,fill opacity=1 ] (318.63,335.88) .. controls (318.63,334.29) and (319.91,333) .. (321.5,333) .. controls (323.09,333) and (324.38,334.29) .. (324.38,335.88) .. controls (324.38,337.46) and (323.09,338.75) .. (321.5,338.75) .. controls (319.91,338.75) and (318.63,337.46) .. (318.63,335.88) -- cycle ;
%Straight Lines [id:da7114623142749121]
\draw [color={rgb, 255:red, 0; green, 0; blue, 0 }  ,draw opacity=1 ]   (298.18,317.5) -- (298.75,380.5) ;

\draw [shift={(298.17,315.5)}, rotate = 89.49] [fill={rgb, 255:red, 0; green, 0; blue, 0 }  ,fill opacity=1 ][line width=0.75]  [draw opacity=0] (8.93,-4.29) -- (0,0) -- (8.93,4.29) -- cycle    ;
%Straight Lines [id:da548053032754751]
\draw [color={rgb, 255:red, 0; green, 0; blue, 0 }  ,draw opacity=1 ]   (355.71,129.5) -- (353.02,259.02) -- (303.13,308.63) ;

\draw [shift={(355.75,127.5)}, rotate = 91.19] [fill={rgb, 255:red, 0; green, 0; blue, 0 }  ,fill opacity=1 ][line width=0.75]  [draw opacity=0] (8.93,-4.29) -- (0,0) -- (8.93,4.29) -- cycle    ;
%Shape: Circle [id:dp3678594235718611]
\draw  [fill={rgb, 255:red, 0; green, 0; blue, 0 }  ,fill opacity=1 ] (385.88,337.38) .. controls (385.88,335.79) and (387.16,334.5) .. (388.75,334.5) .. controls (390.34,334.5) and (391.63,335.79) .. (391.63,337.38) .. controls (391.63,338.96) and (390.34,340.25) .. (388.75,340.25) .. controls (387.16,340.25) and (385.88,338.96) .. (385.88,337.38) -- cycle ;
%Shape: Ellipse [id:dp19146892828979345]
\draw  [dash pattern={on 0.84pt off 2.51pt}] (219,196.5) .. controls (219,181.31) and (231.2,169) .. (246.25,169) .. controls (261.3,169) and (273.5,181.31) .. (273.5,196.5) .. controls (273.5,211.69) and (261.3,224) .. (246.25,224) .. controls (231.2,224) and (219,211.69) .. (219,196.5) -- cycle ;
%Shape: Ellipse [id:dp7189202015970788]
\draw  [dash pattern={on 0.84pt off 2.51pt}] (220,260.5) .. controls (220,245.31) and (232.2,233) .. (247.25,233) .. controls (262.3,233) and (274.5,245.31) .. (274.5,260.5) .. controls (274.5,275.69) and (262.3,288) .. (247.25,288) .. controls (232.2,288) and (220,275.69) .. (220,260.5) -- cycle ;
%Shape: Ellipse [id:dp7423739522942483]
\draw  [dash pattern={on 0.84pt off 2.51pt}] (221,325.5) .. controls (221,310.31) and (233.2,298) .. (248.25,298) .. controls (263.3,298) and (275.5,310.31) .. (275.5,325.5) .. controls (275.5,340.69) and (263.3,353) .. (248.25,353) .. controls (233.2,353) and (221,340.69) .. (221,325.5) -- cycle ;
%Shape: Ellipse [id:dp5201820048660335]
\draw  [dash pattern={on 0.84pt off 2.51pt}] (284,260.5) .. controls (284,245.31) and (296.2,233) .. (311.25,233) .. controls (326.3,233) and (338.5,245.31) .. (338.5,260.5) .. controls (338.5,275.69) and (326.3,288) .. (311.25,288) .. controls (296.2,288) and (284,275.69) .. (284,260.5) -- cycle ;
%Shape: Ellipse [id:dp8569777336792928]
\draw  [dash pattern={on 0.84pt off 2.51pt}] (285,325.5) .. controls (285,310.31) and (297.2,298) .. (312.25,298) .. controls (327.3,298) and (339.5,310.31) .. (339.5,325.5) .. controls (339.5,340.69) and (327.3,353) .. (312.25,353) .. controls (297.2,353) and (285,340.69) .. (285,325.5) -- cycle ;
%Shape: Ellipse [id:dp08000767133134667]
\draw  [dash pattern={on 0.84pt off 2.51pt}] (349,325.5) .. controls (349,310.31) and (361.2,298) .. (376.25,298) .. controls (391.3,298) and (403.5,310.31) .. (403.5,325.5) .. controls (403.5,340.69) and (391.3,353) .. (376.25,353) .. controls (361.2,353) and (349,340.69) .. (349,325.5) -- cycle ;
\draw   (232.4,187) -- (242.9,187)(237.65,182) -- (237.65,192) ;
\draw   (233.89,183.34) -- (241.41,190.66)(241.14,183.41) -- (234.16,190.59) ;
\draw   (251.9,274) -- (262.4,274)(257.15,269) -- (257.15,279) ;
\draw   (253.39,270.34) -- (260.91,277.66)(260.64,270.41) -- (253.66,277.59) ;
\draw   (292.9,273) -- (303.4,273)(298.15,268) -- (298.15,278) ;
\draw   (294.39,269.34) -- (301.91,276.66)(301.64,269.41) -- (294.66,276.59) ;
\draw   (231.4,332.5) -- (241.9,332.5)(236.65,327.5) -- (236.65,337.5) ;
\draw   (232.89,328.84) -- (240.41,336.16)(240.14,328.91) -- (233.16,336.09) ;
\draw   (293.4,312.5) -- (303.9,312.5)(298.65,307.5) -- (298.65,317.5) ;
\draw   (294.89,308.84) -- (302.41,316.16)(302.14,308.91) -- (295.16,316.09) ;
\draw   (355.9,313) -- (366.4,313)(361.15,308) -- (361.15,318) ;
\draw   (357.39,309.34) -- (364.91,316.66)(364.64,309.41) -- (357.66,316.59) ;

% Text Node
\draw (112.5,202) node [scale=0.8] [align=left] {cell 1};
% Text Node
\draw (113.5,261) node [scale=0.8] [align=left] {cell 2};
% Text Node
\draw (112.5,320) node [scale=0.8] [align=left] {cell 3};
% Text Node
\draw (511.5,203) node [scale=0.8] [align=left] {cell 1};
% Text Node
\draw (512.5,262) node [scale=0.8] [align=left] {cell 2};
% Text Node
\draw (511.5,321) node [scale=0.8] [align=left] {cell 3};
% Text Node
\draw (251.75,103.91) node [scale=0.8,rotate=-0.14] [align=left] {cell 1};
% Text Node
\draw (310.75,104.76) node [scale=0.8,rotate=-0.34] [align=left] {cell 2};
% Text Node
\draw (369.75,103.6) node [scale=0.8,rotate=-0.36] [align=left] {cell 3};
% Text Node
\draw (250.75,405.91) node [scale=0.8,rotate=-0.14] [align=left] {cell 1};
% Text Node
\draw (309.75,406.76) node [scale=0.8,rotate=-0.34] [align=left] {cell 2};
% Text Node
\draw (368.75,405.6) node [scale=0.8,rotate=-0.36] [align=left] {cell 3};
% Text Node
\draw (115,151) node  [align=left] {group 1};
% Text Node
\draw (308,54) node  [align=left] {group 2};
% Text Node
\draw (515,151) node  [align=left] {group 3};
% Text Node
\draw (312,459) node  [align=left] {group 4};

\end{tikzpicture}
\caption{Simplified overview of $G$ for $n = 4$. \label{OverviewG}}
\end{figure}
\hfill \break
\noindent It remains to show that there exists a cycle cover in $G$ with cost at most $K$ if and only if there exists a feasible assignment in $\mathcal{I}$ with cost at most $K$.

First, we verify that a cycle cover with finite cost in $G$ corresponds to a feasible assignment of facilities and locations. Note that the connection and relink nodes must be covered by an outer cycle, since any other cycle would induce a cost of $\infty$. Besides the connection and relink node, this cycle contains two nodes that each correspond to an assignment of a different facility. Moreover, these facilities must be assigned to different locations, otherwise this implies an infinite cost. The nodes in a cell that are not covered by an outer cycle must be covered by an inner cycle. Consequently,  nodes on these inner cycles cannot belong to an outer cycle. Therefore, for each group exactly one location is `chosen' to be on an outer cycle, i.e., each facility is assigned to some location. Moreover, no two facilities are assigned to the same location, since this would imply a cost of $\infty$ at the connection node connecting these groups. We conclude that a cycle cover with finite cost corresponds to a feasible assignment and vice versa.

Observe that the objective value of a feasible assignment in the \textsc{QAP} instance equals the total cost of the corresponding cycle cover in the \textsc{QCCP} instance. Namely, the latter cost equals the sum of quadratic costs incurred at the ${n \choose 2}$ connection nodes. If facility $i$ (resp.~$j$) where $i \neq j$ is assigned to location $k'$ (resp.~$l'$) where $k' \neq l'$, then this cost equals $d_{k'l'}w_{ij} + d_{l'k'}w_{ji}$. Taking the sum over all connections, the total cost of the cycle cover equals $\sum_{i =1}^n\sum_{j =1}^n d_{\pi(i)\pi(j)}w_{ij}$ where $\pi : F \rightarrow P$ is the bijection corresponding to the assignment.

We conclude that there exists an assignment for the \textsc{QAP} instance with cost at most $K$ if and only if there exists a feasible cycle cover in the corresponding \textsc{QCCP} instance of cost at most $K$. Since the \textsc{QAP} is strongly $\mathcal{NP}$-hard and the numbers defined in the reduction are polynomially bounded (infinite costs can be replaced by an appropriate value  which is  polynomially bounded in the largest number and the size of $\mathcal{I}$), we conclude that the \textsc{QCCP} is strongly $\mathcal{NP}$-hard. \\ \\
Moreover, as the \textsc{QAP} cannot be approximated within any constant factor \cite{SahniGonzalez} and the reduction above is clearly gap preserving, the result follows.

\end{proof}

\section{The QCCP Linearization Problem} \label{SufficientConditions}
In this section, we formally introduce the linearization problem for the \textsc{QCCP} and derive various sufficient conditions for an instance of the \textsc{QCCP} to be linearizable. Several  of these conditions are used later on to construct lower bounds for the optimal value of the \textsc{QCCP}.
 \\ \\
Let us consider the (linear) cycle cover problem. Given a cost function $p: A \rightarrow \mathbb{R}$, the \textsc{CCP} is the problem of finding a cycle cover of minimum cost. It can be written as follows:
\begin{align} \label{linearCCP}
\min_{x \in \{0,1\}^{m}}\left \{p^Tx \, | \, \, x\in X \right \},
\end{align}
where $X$ is given  in \eqref{X}.
Since the constraint set of $X$ is totally unimodular, it follows that the \textsc{CCP} is solvable in polynomial time. We call an instance $\mathcal{I} = (G,Q)$ of the \textsc{QCCP} linearizable if there exists a cost vector $p \in \mathbb{R}^m$ such that $x^TQx = p^Tx$ for all cycle covers $x \in X$. If such a vector $p$ exists, we call $p$ a linearization vector of $Q$ for the \textsc{QCCP}.

The \textsc{QCCP} linearization problem can be stated as follows: Given an instance $\mathcal{I} = (G,Q)$ of the \textsc{QCCP}, verify whether it is linearizable and, if yes, compute a linearization vector $p$ of $Q$.

In the remaining part of this section we provide sufficient conditions for the quadratic cost matrix $Q$ to be linearizable.
The first type of sufficient conditions for linearizability are related to the constant value property (CVP) for cost vectors or cost matrices. The definition associated with the  \textsc{CCP} is stated below.

\begin{definition} A cost matrix $p$ satisfies the constant value property  if $p^Tx = p^T\bar{x}$ for all cycle covers $x, \bar{x} \in X$.
\end{definition}
A similar definition holds for the quadratic version of the problem.
\begin{definition}
A cost matrix $Q$ satisfies the constant value property if $x^TQx = \bar{x}^TQ\bar{x}$ for all cycle covers $x, \bar{x} \in X$.
\end{definition}
When $Q$ satisfies the constant value property then $Q$ is linearizable, as stated by the following proposition.

\begin{proposition}
Assume that $Q$ satisfies the constant value property, i.e. $x^TQx = \xi$ where $\xi \in \mathbb{R}$ for all $x \in X$, then $Q$ is linearizable with cost vector $p$ defined as $p_e = \xi / n$ for all $e \in A$.
\end{proposition}
\begin{proof}
For all $x \in X$ we have $x^TQx = \xi = n \frac{\xi}{n} = x^Tp$ since $x$ has exactly $n$ nonzero elements.
\end{proof}
A more restricted version of the CVP is obtained when the interaction cost of a single arc  with its successor or predecessor is constant for all cycle covers $x \in X$. We refer to these properties as the row and column constant value property, respectively. These definitions are based on similar definitions by {Punnen et al.\ \cite{Punnen}} for the \textsc{QTSP}.

\begin{definition}
A cost matrix $Q$ satisfies the row CVP if there exists some $b \in \mathbb{R}^m$ such that for all arcs $e \in A$ we have $Q_{ef} = Q_{eg} = b_e$ for all $f, g \in \delta^+(e^-)$ and $Q_{ef} = 0$ otherwise. \\
A cost matrix $Q$ satisfies the column CVP if there exists some $c \in \mathbb{R}^m$ such that for all arcs $e \in A$ we have $Q_{fe} = Q_{ge} = c_e$ for all $f, g \in \delta^-(e^+)$ and $Q_{fe} = 0$ otherwise.
\end{definition}
It is not hard to verify that an instance of the \textsc{QCCP} is linearizable if the cost matrix $Q$ satisfies the row or column CVP.

\begin{proposition}
If $Q$ satisfies the row CVP or the column CVP, then $Q$ is linearizable.
\end{proposition}

\begin{proof}
We prove the case when $Q$ satisfies the row CVP.
Assume that  $b \in \mathbb{R}^m$ is such that  for all arcs $e \in A$,  $Q_{ef} = Q_{eg} = b_e$ for all $f, g \in \delta^+(e^-)$ and $Q_{ef} = 0$ otherwise.
 Since $Q_{ef} = 0$ when $e$ and $f$ are not successors, we know that $x^TQx = \sum_{e \in A}\sum_{f \in \delta^+(e^-)}Q_{ef}x_ex_f$. We have
\begin{align*}
\sum_{e \in A}\sum_{f \in \delta^+(e^-)}Q_{ef}x_ex_f = \sum_{e \in A}x_eb_e\sum_{f \in \delta^+(e^-)}x_f = \sum_{e \in A}x_eb_e = x^Tb.
\end{align*}
The proof for the column CVP is similar.
\end{proof}

\noindent A matrix $Q \in \mathbb{R}^{m \times m}$ is called a sum matrix if there exist $b, c \in \mathbb{R}^m$ such that $Q_{ef} = b_e + c_f$ for all $e, f$. A weak sum matrix is a matrix for which this property holds except for the entries on the diagonal, i.e. $Q_{ef} = b_e + c_f$ for all $e \neq f$. The weak sum property is used as a condition for linearizability for several quadratic problems, see e.g., \cite{HuSotirov1, Punnen}. Since in this work we only incur a cost when two arcs are successive, we use a different form of the weak sum condition in which we only put a restriction on successive arcs. We call this condition the incident weak sum property.

\begin{definition} A matrix $Q$ is called incident weak sum if there exist vectors $b, c \in \mathbb{R}^m$ such that $Q_{ef} = b_e + c_f$ for all $e \in A$, $f \in \delta^+(e^-)$ and $Q_{ef} = 0$ otherwise, i.e. $Q_{ef} = b_e + c_f$ for all pairs of arcs $e,f$ such that $f$ is a successor of $e$. If such vectors $b$ and $c$ exist, these vectors are called supporting vectors of $Q$. \label{DefIncidentWeakSum}
\end{definition}
If the quadratic cost matrix $Q$ is an incident weak sum matrix, then $Q$ is linearizable as stated in the following proposition.
\begin{proposition} \label{weaksum}
Let $Q$ be an incident weak sum matrix with supporting vectors $b, c \in \mathbb{R}^m$. Then $Q$ is linearizable with cost vector $p = b + c$.
\end{proposition}
\begin{proof}
We show that for all $x \in X$ we have $x^TQx = p^Tx$ where $p = b + c$. Note that since $Q_{ef} = 0$ for all arcs that are not successors, we have $x^TQx = \sum_{e \in A} \sum_{f \in \delta^+(e^-)}Q_{ef}x_ex_f$. Now,
\begin{align*}
\sum_{e \in A}\sum_{f \in \delta^+(e^-)}Q_{ef}x_ex_f  & = \sum_{e \in A}\sum_{f \in \delta^+(e^-)}(b_e + c_f)x_ex_f\\
& = \sum_{e \in A}b_ex_e \sum_{f \in \delta^+(e^-)}x_f + \sum_{f \in A}c_fx_f\sum_{e \in \delta^-(f^+)}x_e \\
& = \sum_{e \in A}b_ex_e  + \sum_{f \in A}c_fx_f  = \sum_{e \in A}p_ex_e.
\end{align*}
Here we use the fact that $\sum_{f \in \delta^+(e^-)}x_f = \sum_{e \in \delta^-(f^+)}x_e = 1$, since $x$ is a cycle cover.
\end{proof}
From Proposition \ref{weaksum} it follows that the incident weak sum property is a sufficient condition for $Q$ to be linearizable. By including linear arc costs, this result remains valid, since we only increase the entries on the diagonal of $Q$. \\
Moreover, note that when $Q$ satisfies the row or column CVP, then $Q$ is an incident weak sum matrix.
Next, we provide a special type of instance for which the cost matrix is not by definition linearizable, but for which we can still obtain its optimal value by solving a linear cycle cover problem.

\begin{definition} A matrix $Q \in \mathbb{R}^{m \times m}$ is called a symmetric product matrix if $Q = aa^T$ for some vector $a \in \mathbb{R}^m$.
\end{definition}
Equivalently, we can say that $Q$ is a symmetric product matrix if it is a symmetric positive semidefinite matrix of rank one. Instances with such a quadratic cost matrix can be solved in polynomial time, as stated in the following proposition.
\begin{proposition}
If the quadratic cost matrix $Q$ is a symmetric product matrix, i.e. $Q = aa^T$ for some $a \in \mathbb{R}^m$, then the optimal cycle cover can be computed in polynomial time.
\end{proposition}
\begin{proof}
Let $Q$ be such that $Q = aa^T$ for some $a \in \mathbb{R}^m$. Then,
\begin{align*}
x^TQx = x^Taa^Tx = (a^Tx)^T(a^Tx) = (a^Tx)^2.
\end{align*}
Minimizing $x^TQx$ over all $x \in X$ is then equivalent to minimizing $a^Tx$ over all cycle covers $x \in X$.
\end{proof}

\noindent Until now we considered instances for which $Q$ is of the desired type, i.e., $Q_{ef} = 0$ when $f$ is not a successor of $e$. Below we derive two sufficient conditions for linearizability of a matrix $Q$ which can have nonzero interaction cost between non-consecutive arcs. Although these cost matrices do not meet the assumptions of the \textsc{QCCP}, we can still use them to derive strong bounds for the objective value of the original problem. This is addressed in Section \ref{LinearizationBasedBounds}. \\ \\
Punnen et al.\ \cite{Punnen} introduce  a generalized version of the weak sum property for the \textsc{QTSP}. Their approach can be applied to the \textsc{QCCP}. However, since Punnen et al.\ \cite{Punnen} prove the condition to hold for complete graphs, we provide a proof for general digraphs.

First, we define some new terminology. Instead of writing $Q_{ef}$ for $e, f \in A$ we can also write $Q_{ij,kl}$ with $(i,j), (k,l) \in A$. Let $N^+_i$ denote the set of nodes $j$ for which there exists an arc $(i,j) \in A$, i.e., $N^+_i := \{ j \in N \, | \, \, (i,j) \in A \}$. Similarly, let $N^-_i$ be the set of nodes $j$ for which an arc $(j,i) \in A$ exists, i.e., $N^-_i := \{ j \in N \, | \, \, (j,i) \in A\}$. Now we can introduce the notion of a generalized weak sum matrix and prove that it is linearizable.

\begin{definition} \label{generalizedweaksumdef} $Q$ is called a generalized weak sum matrix if there exist $B, C \in \mathbb{R}^{m \times n}$ and $D, T  \in \mathbb{R}^{n \times m}$ such that $Q_{ij,kl} = b_{ij,k} + c_{ij,l} + d_{i,kl} + t_{j,kl}$ for all $i,j,k,l$ with $(i,j), (k,l) \in A$. If such $B, C, D$ and $T$ exist, these matrices are called supporting matrices of $Q$.
\end{definition}
Now we can prove the following proposition.
\begin{proposition} \label{generalizedweaksum}
If $Q$ is a generalized weak sum matrix with supporting matrices $B, C \in \mathbb{R}^{m \times n}$ and $D, T  \in \mathbb{R}^{n \times m}$, then $Q$ is linearizable with cost vector $p$ given by $p_{ij} = \sum_{k = 1}^nb_{ij,k} +  \sum_{k = 1}^nc_{ij,k} + \sum_{k = 1}^nd_{k,ij} + \sum_{k = 1}^nt_{k,ij}$.
\end{proposition}
\begin{proof}
Let $\bar{b}_{ij} := \sum_{k = 1}^nb_{ij,k}$, $\bar{c}_{ij} := \sum_{k = 1}^nc_{ij,k}$, $\bar{d}_{ij} := \sum_{k = 1}^nd_{k,ij}$, $\bar{t}_{ij} := \sum_{k = 1}^nt_{k,ij}$ and $p_{ij} = \bar{b}_{ij} + \bar{c}_{ij} + \bar{d}_{ij} + \bar{t}_{ij}$ for all $(i,j) \in A$. Then, for $x \in X$,
\begin{align*}
x^TQx & = \sum_{i \in N}\sum_{j \in N^+_i}\sum_{k \in N}\sum_{l \in N^+_k}Q_{ij,kl}x_{ij}x_{kl} \\
& = \sum_{i \in N}\sum_{j \in N^+_i}\sum_{k \in N}\sum_{l \in N^+_k}b_{ij,k}x_{ij}x_{kl} + \sum_{i \in N}\sum_{j \in N^+_i}\sum_{l \in N}\sum_{k \in N^-_l}c_{ij,l}x_{ij}x_{kl} \\
& \quad + \sum_{k \in N}\sum_{l \in N^+_k}\sum_{i \in N}\sum_{j \in N^+_i}d_{i,kl}x_{ij}x_{kl} + \sum_{k \in N}\sum_{l \in N^+_k}\sum_{j \in N}\sum_{i \in N^-_j}t_{j,kl}x_{ij}x_{kl}\\
& = \sum_{i \in N}\sum_{j \in N^+_i}x_{ij} \sum_{k \in N}b_{ij,k} \sum_{l \in N^+_k}x_{kl} +  \sum_{i \in N}\sum_{j \in N^+_i} x_{ij}\sum_{l \in N}c_{ij,l}\sum_{k \in N^-_l}x_{kl} \\
& \quad + \sum_{k \in N}\sum_{l \in N^+_k}x_{kl}\sum_{i \in N}d_{i,kl}\sum_{j \in N^+_i}x_{ij} + \sum_{k \in N}\sum_{l \in N^+_k}x_{kl}\sum_{j \in N}t_{j,kl}\sum_{i \in N^-_j}x_{ij}\\
& = \sum_{i \in N}\sum_{j \in N^+_i}(\bar{b}_{ij} + \bar{c}_{ij} + \bar{d}_{ij} + \bar{t}_{ij})x_{ij} = \sum_{i \in N}\sum_{j \in N^+_i}p_{ij}x_{ij} \, ,
\end{align*}
\noindent where we use the fact that $\sum_{l \in N^+_k}x_{kl} = \sum_{k \in N^-_l}x_{kl} = \sum_{j \in N^+_i}x_{ij} = \sum_{i \in N^-_j}x_{ij} = 1$ since $x$ is a cycle cover.
\end{proof}
Note that an incident weak sum matrix can be seen as a special case of a generalized weak sum matrix. That is, for all $(i,j) \in A$ we set $b_{ij,j} = b_{ij}$ and $b_{ij,k} = 0$ for all $k \neq j$ and  for all $(k,l) \in A$ we set $t_{k,kl} = t_{kl}$ and $t_{j,kl} = 0$ for all $j \neq k$. Moreover, let $C$ and $D$ be the zero matrix. Then we have $Q_{ij,jl} = b_{ij,j} + c_{ij,l} + d_{i,jl} + t_{j,jl} = b_{ij} + t_{jl}$ for all $(i,j), (j,l) \in A$ and $Q_{ij,kl} = 0$ otherwise.

When $Q$ is a generalized weak sum matrix, we need $4mn$ parameters to describe $Q$. This number can be reduced by considering a more restricted version of  a generalized weak sum matrix.
\begin{definition} $Q$ is called a restricted generalized weak sum matrix if there exist $C \in \mathbb{R}^{m \times n}$, $D  \in \mathbb{R}^{n \times m}$ and $b, t \in \mathbb{R}^m$ such that $Q_{ij,jl} = b_{ij} + c_{ij,l} + d_{i,jl} + t_{jl}$ for all $i,j,l$ with $(i,j), (j,l) \in A$ and $Q_{ij,kl} = c_{ij,l} + d_{i,kl}$ otherwise. If such $C, D$ and $b, t$ exist, these are called supporting matrices and vectors, respectively.
\end{definition}
We can show that restricted generalized weak sum matrices are linearizable.
\begin{proposition} \label{generalizedweaksumRestrict}
If $Q$ is a restricted generalized weak sum matrix with supporting matrices $C, D$ and supporting vectors $b,t$, then $Q$ is linearizable with vector $p$ given by $p_{ij} = b_{ij} +  \sum_{k=1}^nc_{ij,k} + \sum_{k=1}^nd_{k,ij} + t_{ij}$.
\end{proposition}
\begin{proof}
Define $B \in \mathbb{R}^{m \times n}$ and $T \in \mathbb{R}^{n \times m}$ as follows:
\begin{align*}
B_{ij,k} & = \begin{cases} b_{ij} \quad & \text{ if $k = j$} \\ 0 & \text{ otherwise} \end{cases} \qquad \qquad \qquad \forall (i,j) \in A, \,  \forall k \in N \\
T_{k,ij} & = \begin{cases} t_{ij} \quad & \text{ if $k = i$} \\ 0 & \text{ otherwise} \end{cases} \qquad \qquad \qquad \forall (i,j) \in A, \, \forall k \in N
\end{align*}
Now the matrices $B, C, D$ and $T$ are such that they satisfy the conditions of Proposition \ref{generalizedweaksum}. This implies that $Q$ is linearizable with vector $p'$ given by $p'_{ij} = \sum_{k = 1}^nb_{ij,k} +  \sum_{k = 1}^nc_{ij,k} + \sum_{k = 1}^nd_{k,ij} + \sum_{k = 1}^nt_{k,ij}$. Since $\sum_{k = 1}^nb_{ij,k} = b_{ij}$ and $\sum_{k = 1}^nt_{k,ij} = t_{ij}$ it follows that $Q$ is linearizable with vector $p := b_{ij} +  \sum_{k = 1}^nb_{ij,k} + \sum_{k = 1}^nd_{k,ij} + t_{ij}$.
\end{proof}

\section{Linearization Based Bounds for \textsc{QCCP}} \label{LinearizationBasedBounds}
In this section we show how the sufficient conditions for linearizability can be used to derive bounds for the optimal value of the \textsc{QCCP}. The construction of these bounds is provided in Section \ref{LBBConstruction}. Section \ref{Prelim1} shows some preliminary numerical results of these bounding procedures.

\subsection{Construction of Linearization Based Bounds} \label{LBBConstruction}
When an instance of the \textsc{QCCP} is linearizable, we can solve the problem in polynomial time by solving the corresponding linear cycle cover problem. When a quadratic cost matrix $Q$ is not linearizable, we can still use the sufficient conditions for linearizability to find lower bounds for the optimal value of the problem. This approach is introduced by Hu and Sotirov \cite{HuSotirov2} for general binary quadratic problems. We here use tailor made sufficient conditions for the \textsc{QCCP}, which leads to efficient lower bounds as we show later in the numerical results.

Before we proceed, let us recall the linear cycle cover problem. We introduce the matrix $U \in \mathbb{R}^{n \times m}$ with $U_{ie} = 1$ if node $i$ is the starting node of arc $e$ and 0 otherwise. Similarly, we define $V \in \mathbb{R}^{n \times m}$ with $V_{ie} = 1$ if node $i$ is the ending node of arc $e$ and 0 otherwise. Since the matrix $[U^T, V^T]^T$ is totally unimodular, the optimal value of the \textsc{CCP} using cost vector $p$ equals
\begin{align}
OPT(p)  & := \min_{x \in \mathbb{R}^m} \{p^Tx \, | \, \, \begin{bmatrix}
U \\ V
\end{bmatrix} x = \mathbbm{1}_{2n} \, , \, \,  x \geq 0 \} \label{CCPprimal} \\
& \,\, = \max_{y \in \mathbb{R}^{2n}}\{\mathbbm{1}_{2n}^T y \, | \, \, [U^T, V^T]y \leq p \}, \label{CCPdual}
\end{align}
where $\mathbbm{1}_{2n} \in \mathbb{R}^{2n}$ equals the vector of ones. Note that (\ref{linearCCP}) and (\ref{CCPprimal}) are equivalent optimization problems. The corresponding dual problem is given in (\ref{CCPdual}).

When $Q$ is linearizable with linearization vector $p$, we can find the optimal value for the \textsc{QCCP} by computing $OPT(p)$ using (\ref{CCPprimal}) or (\ref{CCPdual}). If $Q$ is not linearizable, we can search for a linearizable matrix $\hat{Q}$ that is as close as possible to $Q$. To guarantee that $\hat{Q}$ is indeed linearizable, it should satisfy one of the sufficient conditions for linearizability derived in Section \ref{SufficientConditions}. We define the sets $S_i(Q)$, for $i = 1, 2, 3$, consisting of cost matrices $\hat{Q}$ such that $\hat{Q}$ is linearizable w.r.t.\ a sufficient condition for linearizability and $Q - \hat{Q}$ is elementwise nonnegative. We have
\begin{align*}
S_1(Q) & := \{\hat{Q} \in \mathbb{R}^{m \times m} \, | \, \, \hat{Q} \text{ {is an incident weak sum matrix} and } Q - \hat{Q} \geq 0 \}, \\
S_2(Q) & := \{\hat{Q} \in \mathbb{R}^{m \times m} \, | \, \, \hat{Q} \text{ {is a restricted generalized weak sum matrix} and } Q - \hat{Q} \geq 0 \}, \\
S_3(Q) & := \{\hat{Q} \in \mathbb{R}^{m \times m} \, | \, \, \hat{Q} \text{ {is a generalized weak sum matrix} and } Q - \hat{Q} \geq 0 \}.
\end{align*}
\begin{remark} We do not consider the sets of cost matrices $Q$ satisfying the row or column CVP, since these are special types of incident weak sum matrices. These type of matrices are contained in $S_1$.
\end{remark}

$S_i(Q)$ can be seen as the set of all the linearizable cost matrices of a specific type that are suitable for obtaining lower bounds for the optimal value of the problem. For this purpose, we define for $i = 1, 2, 3$ the set $\tau_i(Q)$ of cost vectors $\hat{p} \in \mathbb{R}^m$ as
\begin{align*}
\tau_i(Q) := \{\hat{p} \in \mathbb{R}^m \, | \, \, x^T\hat{p} = x^T\hat{Q}x \text{ for all } x \in X \, , \, \, \hat{Q} \in S_i(Q)\}.
\end{align*}
It is clear that for all $i$ and all $\hat{p} \in \tau_i(Q)$ we have
\begin{align*}
OPT(Q) = \min_{x \in X}\{x^TQx\} \geq \min_{x \in X}\{x^T\hat{Q}x\} = \min_{x \in X}\{x^T\hat{p}\} = OPT(\hat{p}).
\end{align*}
So, indeed, $OPT(\hat{p})$ is a lower bound for the optimal objective value of the \textsc{QCCP} for all $i$ and $\hat{p} \in \tau_i(Q)$. By maximizing over all cost vectors in $\tau_i(Q)$, we obtain the strongest linearization based bound with respect to the set $S_i(Q)$, which we denote by $v_{LBB}^i$,  see also \cite{HuSotirov2}:
\begin{align} \label{LBB}
v^i_{LBB} := \max_{\hat{p} \in \tau_i(Q)}\{OPT(\hat{p})\} = \max_{\substack{y \in \mathbb{R}^{2n}\\
\hat{p} \in \mathbb{R}^m}}\{\mathbbm{1}_{2n}^T y \, | \, \, [U^T, V^T]y \leq \hat{p} \, , \, \, \hat{p} \in \tau_i(Q) \}.
\end{align}
\noindent The corresponding bounding approaches are denoted by $LBB1$, $LBB2$ and $LBB3$, respectively.

\begin{remark} Recall that the matrices in $S_2(Q)$ and $S_3(Q)$ can have nonzero interaction cost for non-consecutive arcs, so they do not satisfy the assumptions on the quadratic cost matrix of the \textsc{QCCP}. Nevertheless, they can still be used to derive lower bounds for the original problem. In other words, we search for the general quadratic cost matrix that is linearizable and as close as possible to $Q$ that gives us the best lower bound.
\end{remark}

As long as the set $\tau_i(Q)$ is a polyhedron, the corresponding bound $v^i_{LBB}$ can be calculated by solving the linear programming problem (\ref{LBB}). The sets $\tau_i(Q)$ for $i = 1, 2, 3$ are indeed nonempty polyhedra, since they can be described by a finite number of linear equalities and inequalities. These polyhedral descriptions are provided in Table \ref{polyhedralDescriptions}.

\begin{table}[H]
\centering
\scriptsize
\begin{tabular}{@{}ccll@{}}
\toprule
\textbf{Set} & \textbf{\begin{tabular}[c]{@{}c@{}}Type of \\ cost matrix\end{tabular}}                                                      & \multicolumn{2}{c}{\textbf{(In)equalities which describe the set}}                                                                                                                                                                                                                                                                                                                                                                                                   \\ \midrule
$\tau_1(Q)$  & \begin{tabular}[c]{@{}c@{}}Incident \\ weak sum \\ matrix\end{tabular}              & \begin{tabular}[c]{@{}l@{}}$b_e + c_f \leq Q_{ef}$\\ $\hat{p}_e = b_e + c_e$\\ $b, c \in \mathbb{R}^m$\end{tabular}                                                                                                                                                                                            & \begin{tabular}[c]{@{}l@{}}$\forall e \in A, f \in \delta^+(e^-)$\\ $\forall e \in A$\\ $\quad$\end{tabular}                                      \\ \midrule
$\tau_2(Q)$  & \begin{tabular}[c]{@{}c@{}}Restricted \\ generalized \\ weak sum \\ matrix\end{tabular} & \begin{tabular}[c]{@{}l@{}}$b_{ij} + c_{ij,l} + d_{i, jl} + t_{jl} \leq Q_{ij,jl}$\\ $c_{ij,l} + d_{i,kl} \leq  Q_{ij,kl}$\\ $\hat{p}_{ij} = b_{ij} +  \sum_{k=1}^nc_{ij,k} + \sum_{k=1}^nd_{k,ij} + t_{ij}$\\ $b, t \in \mathbb{R}^m, C \in \mathbb{R}^{m \times n}, D \in \mathbb{R}^{n \times m}$\end{tabular} & \begin{tabular}[c]{@{}l@{}}$\forall (i,j), (j,l) \in A$\\ $\forall (i,j), (k,l) \in A, j \neq k$\\ $\forall (i,j) \in A$ \\ $\quad$\end{tabular} \\
\midrule
$\tau_3(Q)$  & \begin{tabular}[c]{@{}c@{}}Generalized \\ weak sum \\matrix\end{tabular}            & \begin{tabular}[c]{@{}l@{}}$b_{ij,k} + c_{ij,l} + d_{i,kl} + t_{j,kl} \leq Q_{ij,kl}$\\ $\hat{p}_{ij} = \sum_{k=1}^nb_{ij,k} +  \sum_{k=1}^nc_{ij,k} + \sum_{k=1}^nd_{k,ij} + \sum_{k =1}^nt_{k,ij}$ \\ $B, C \in \mathbb{R}^{m \times n}, D, T \in \mathbb{R}^{n \times m}$\end{tabular}                       & \begin{tabular}[c]{@{}l@{}}$\forall (i,j), (k,l) \in A$\\ $\forall (i,j) \in A$\\ $\quad$\end{tabular}                                            \\ \bottomrule
\end{tabular}
\caption{Polyhedral descriptions of the sets $\tau_1(Q), \tau_2(Q)$ and $\tau_3(Q)$. \label{polyhedralDescriptions}}
\end{table}
By construction, we have $\tau_1(Q) \subseteq \tau_2(Q) \subseteq \tau_3(Q)$ for all quadratic cost matrices $Q$. Consequently, we can establish the following result about the quality of the corresponding linearization based bounds.
\begin{theorem}
For all instances of the \textsc{QCCP}, we have $v^1_{LBB} \leq v^2_{LBB} \leq v^3_{LBB}$.
\end{theorem}
\begin{proof}
The proof follows immediately from the construction of the linearization based bounds and the definitions of the incident weak sum matrix, the restricted generalized weak sum matrix and the generalized weak sum matrix.
\end{proof}

\noindent Hu and Sotirov \cite{HuSotirov2} argue that the linearization based bounds can be improved by extending the sets $\tau_i(Q)$ using a skew symmetric matrix $M$. That is, since each skew symmetric matrix is linearizable, a matrix $\hat{Q}$ is linearizable if and only if $\hat{Q} + M$ is linearizable for all $M$ with $M + M^T = 0$. Using this, the set $\tau_1(Q)$ can be extended to:
\begin{align}
\tau^{skew}_1(Q) :=& \, \, \left\{ \hat{p} \in \mathbb{R}^m \quad  \vline
\begin{aligned}
 \quad b_e + c_f + M_{ef} \leq Q_{ef},~\forall e \in A, f \in \delta^+(e^-)\\
 \hat{p}_e = b_e + c_e,~\forall e \in A \\
 M_{ef} = 0, \quad \forall e \in A, f \notin \delta^+(e^-) \\
b, c \in \mathbb{R}^m, M \in \mathbb{R}^{m \times m}, \, M + M^T = 0 \end{aligned} \right\} \label{skew}
 \end{align}
Note that in $\tau^{skew}_1(Q)$ we only include skew symmetric matrices whose support corresponds to the pairs of successive arcs in $G$, since adding dense skew symmetric matrices would increase computational complexity.
Since $\tau_1(Q) \subseteq \tau_1^{skew}(Q)$, it follows that we can obtain a stronger bound by maximizing over $\tau_1^{skew}(Q)$, see Section \ref{NumericalResults}. The same extension can be applied to any set $\tau_i(Q)$.

\subsection{Preliminary Results} \label{Prelim1}
\noindent In order to check the quality of the bounds derived above, we perform a preliminary numerical study. We create instances according to the $G(n,p)$ Erd\H os-R\' enyi model \cite{ErdosRenyi}. Here $n$ equals the number of nodes and $p$ equals the probability that an arc is included. We create instances for various values of $n$ and $p$. The interaction cost between any two {successive} arcs is drawn uniformly at random as an integer from $\{1, ..., 100\}$.
In Table \ref{PreliminaryExperiments} we present the bounds $v_{LBB}^1, v_{LBB}^2$ and $v_{LBB}^3$ and  the computational times required for computing them.
Experiments are performed using a pc with an Intel(R) Core(TM) i5-6500 CPU, 3.20 GHz and 8 GB memory using CPLEX 12.6 as the solver.
The maximum computation time is set to 3600 seconds and we put `n.a.' in the table when this maximum is reached before a solution is obtained.

\begin{table}[H]
\footnotesize
\centering
\begin{tabular}{@{}cccccccccc@{}}
\toprule
    &     &   &    & \multicolumn{2}{c}{$LBB1$} & \multicolumn{2}{c}{$LBB2$} & \multicolumn{2}{c}{$LBB3$} \\ \cmidrule(l){5-6} \cmidrule(l){7-8} \cmidrule(l){9-10}
$p$ & $n$ & $m$ &  $OPT$ & bound          & time & bound          & time                     & bound          & time                    \\ \midrule
0.1 &  20 &44 & 923 & 923 & 0.047 & 923 & 0.264 & 923 & 0.051 \\
 & 25& 76 & 1039 & 971 & 0.005 & 971 & 0.477 & 999 & 1.227 \\
 &30 & 100 & 1082 & 1066 & 0.013    & 1066 & 0.899    &1082  & 0.787   \\
0.3 & 15 & 61 & 485   & 478          & 0.010          & 478          & 0.635          & 485            & 0.714         \\
    & 20 & 118 & 438   & 377          & 0.031          & 384          & 1.265          & 390         & 77.21         \\
    & 25 &  172 & 382   & 291          & 0.050          & 295         & 2.531          & 295          & 869.0         \\
0.5 & 15 &  116 & 226   & 215          & 0.034          & 215          & 1.407          & 215          & 90.39         \\
    & 20 & 177 &    255   &    189            &     0.059                   & 190	 &       12.05         &     190           &             2105 \\
    & 25 & 306 &   n.a.   &    172   &   0.516    &               173 &          353.6      &      n.a.          &              3600 \\
0.7 & 15 & 149 & 173& 127 & 0.017 & 128 & 0.986 & 128 & 580.5 \\
& 20 & 263& n.a. & 116 & 0.094 & 116 & 7.778 & n.a. & 3600 \\
 & 25& 396& n.a. & 129 & 0.194 & 129 & 18.26 & n.a. & 3600  \\
    \bottomrule
\end{tabular}
\caption{Bounds and computation times in seconds of linearization based bounds on Erd\H os-R\' enyi instances. \label{PreliminaryExperiments}}
\end{table}

\noindent By construction, the optimal solution has always an integer objective value. Therefore, we round up all bounds to the nearest integer.
The results of Table \ref{PreliminaryExperiments} show that the linearization based bounds $LBB1$, $LBB2$ and $LBB3$ do not differ  significantly,
especially for the larger instances. At the same time, the computation times differ significantly. It turns out that $LBB1$ is most efficient.
Therefore, this bound can be preferred when taking both quality and efficiency into account.

\section{Reformulated LBB Approach} \label{ReformulatedLBBApproach}
In this section we discuss how a reformulation of the quadratic cost matrix can be used to obtain a non-decreasing sequence of lower bounds that are based on the linearization based bound. It is important to note that one can construct such a bounding procedure using any sufficient condition for linearizability, not only the ones discussed in Section \ref{LinearizationBasedBounds}. \\ \\
Suppose we are given a sufficient condition for linearizability. Let $S(Q)$ and $\tau(Q)$ be as in Section \ref{LinearizationBasedBounds}, but now for a general sufficient condition. That is, $S(Q)$ is the set consisting of all linearizable cost matrices $\hat{Q}$ of this type with $\hat{Q} \leq Q$ and $\tau(Q)$ consists of the corresponding linearization vectors. Moreover, we assume that the set $\tau(Q)$ is a polyhedral set. Let $Q_0$ be the initial quadratic cost matrix. If $\hat{Q}_0 \in S(Q_0)$, we know there exists some $p_1 \in \tau(Q_0)$ such that $x^T\hat{Q}_0x = x^Tp_1$ for all $x \in X$. This leads to the following reformulation of the objective function
\begin{align}
x^TQ_0x = x^T\hat{Q}_0x + x^T(Q_0 - \hat{Q}_0)x = x^Tp_1 + x^T(Q_0 - \hat{Q}_0)x
\end{align}
for all cycle covers $x \in X$. By letting $p_0$ be the $m \times 1$ zero vector and $Q_1 := Q_0 - \hat{Q}_0$, this relation can be written as $x^Tp_0 + x^TQ_0x = x^Tp_1 + x^TQ_1x$ for all $x \in X$. The vector $p_1$ is taken to be the largest linearization of the matrix $Q_0$ (see (\ref{LBB})) with $Q_1$ being the residual quadratic part. Recall that the bound $v_{LBB}$ is calculated by only taking the linear part of this new objective function into account. By construction we have $x^Tp_0 \leq x^Tp_1$ for all $x \in X$ and $0 \leq Q_1 \leq Q_0$.

Now we can proceed in a similar way by considering the linearization problem of the residual cost matrix $Q_1$. In other words, we search for a linearizable matrix in $S(Q_1)$ and its corresponding linearization vector $p_2$. If we let $Q_2$ be the new residual matrix, then the objective function can be reformulated as $x^Tp_0 + x^TQ_0x = x^T(p_1 + p_2) + x^TQ_2x$. Since $x^Tp_1 \leq x^T(p_1 + p_2)$, a stronger bound can be obtained. This procedure can be repeated to obtain a sequence of bounds. However, it is in general not possible to find a vector $p_2$ for which this bound has strictly improved. That is, since $p_1 + p_2 \in \tau(Q_0)$ this would imply that $p_1$ is not the optimal solution to (\ref{LBB}). Thus, applying this procedure iteratively, the resulting sequence of bounds remains constant after the first iteration. To overcome this issue, we need to reformulate the residual cost matrix in each step. \\ \\
In the literature, various iterative bounding procedures have been proposed, see e.g., \cite{Burkard, PunnenPandey, RostamiMalucelli, Rostami}. In this paper we introduce a new approach that is different in two ways. First, the existing bounding procedures are mainly based on the classical Gilmore-Lawler type bound. Our approach is based on general sufficient conditions for linearizability and we can show that the Gilmore-Lawler type bounding procedure is a special case of this approach, see Section \ref{GilmoreLawler}. Second, the existing bounding procedures mostly use a fixed reformulation of the cost matrix in each iteration.
However, using a fixed reformulation is in general not the best one can do. Here, we search for the reformulation of the cost matrix that results in the strongest bound in the next iteration.
 For this purpose, we define the notion of an equivalent representation of a matrix, see e.g.\ \cite{PunnenPandey}.

\begin{definition} Let $(G, Q)$ be an instance of the \textsc{QCCP}. Then $(G,R)$ is an equivalent representation of $(G,Q)$ if $x^TQx = x^TRx$ for all $x \in X$.
\end{definition}

\noindent If there is no confusion about the graph $G$ under consideration, we say that $R$ is an equivalent representation of $Q$. It is easy to verify that a matrix $R$ is an
equivalent representation of $Q$ if for all $e,f \in A$ we have $R_{ef} + R_{fe} = Q_{ef} + Q_{fe}$. Here,  we focus on a specific type of equivalent representation,
which we call an $\eta$-equivalent representation of $Q$.
\begin{definition}
Given $\eta \in [0,1]$, an equivalent representation $Q^\eta := \eta Q + (1 - \eta)Q^T$ of $Q$ is called an $\eta$-equivalent representation.
\end{definition}
In other words, an $\eta$-equivalent representation is obtained by taking a convex combination of $Q$ and $Q^T$.
Moreover, it follows that if $R$ and $Q$ are equivalent representations, then a linearization of $R$ is also a linearization of $Q$ and vice versa.

Instead of considering the linearization problem of the residual matrix $Q_1$, we can consider the linearization problem of $Q_1^\eta$ for some $\eta \in [0,1]$. Since $Q_1^\eta$ has a different structure than $Q_1$, it is in general possible to find a linearizable matrix $\hat{Q}_1 \in S(Q_1^\eta)$ and a corresponding linearization vector that result in a strictly stronger bound.\\

As already mentioned above, many approaches in the literature are based on taking a fixed value for $\eta$, e.g.\ $\eta = \frac{1}{2}$ which corresponds to the case of symmetrizing. This does not give the best bound in general.
 Instead, we search for  $\eta \in [0,1]$ that results in the strongest bound in each iteration. Suppose we are in step $k$ of the algorithm in which we consider the linearization problem of the residual matrix $Q_{k - 1}$. Then the optimal equivalent representation of $Q_{k -1}$ and its corresponding vector $p_k \in \tau(Q^\eta_{k - 1})$ can be computed simultaneously by solving the following problem:
\begin{align} \label{equivalentrepresent}
r_k := \max_{\substack{y \in \mathbb{R}^{2n}\\ p_k \in \mathbb{R}^m \\ \eta \in [0,1]}}\{\mathbbm{1}_{2n}^Ty \, | \, \, [U^T, V^T]y \leq p_k, \, p_k \in \tau(Q_{k-1}^\eta) \},
\end{align}
which equals the additional amount of quadratic cost that can be linearized in iteration $k$.
Note that if the set $\tau(Q_{k - 1})$ is a polyhedron, then  $\tau(Q_{k-1}^\eta)$ is also a polyhedron and the corresponding  problem  (\ref{equivalentrepresent}) can be solved in polynomial time.
For the sufficient conditions mentioned in Section \ref{LinearizationBasedBounds} this is indeed the case. \\ \\
Finally, we provide a new bounding procedure that is based on iteratively finding the best equivalent representation of the residual cost matrix and its optimal linearizable matrix. Starting with $Q_0 = Q$, the goal is to find the best linearizable matrix $\hat{Q}_{k-1}$ of an equivalent representation of the residual matrix $Q_{k-1}$ and its corresponding linearization vector $p_k$. In each iteration we compute $r_k$ by (\ref{equivalentrepresent}) and let
$d_k = d_{k - 1} + p_k$ which equals the total linearization vector of $Q$. The final bound is given by the sum of all $r_k$'s, which we call the reformulation based bound.
The procedure is given in Algorithm \ref{AlgorithmLBB}.

\begin{algorithm}[H]
\small
\caption{$LBB$ Reformulation Algorithm}\label{AlgorithmLBB}
\begin{algorithmic}[1]
\State $Q_0 = Q$, $d_0 = {\mathbf 0}$, $k = 1$, $r_0 = \infty$
\While {$r_{k-1} >0$}
\State Compute $r_k, p_k$ and $\eta$ using (\ref{equivalentrepresent}).
\State Construct the linearizable matrix $\hat{Q}_{k-1}$ using the optimal solution of (\ref{equivalentrepresent}). \Comment{See Remark \ref{RemarkAlgorithm}}
\State $Q_k \gets \eta Q_{k -1} + (1 - \eta)Q_{k - 1}^T - \hat{Q}_{k - 1}$
\State $d_k \gets d_{k - 1} + p_k$
\State $k \gets k + 1$
\EndWhile
\State $v_{RBB} = \sum_{i = 1}^{k-1}r_i$ \\
\Return $d_k$, $v_{RBB}$
\end{algorithmic}
\end{algorithm}

\begin{remark} \label{RemarkAlgorithm}
Note that steps 3 and 4 of Algorithm 1 depend on the specific sufficient condition for linearizability.
For instance, for the incident weak sum condition we construct  in step 4 the linearizable matrix $\hat{Q}_{k-1}$ in the following way
$(\hat{Q}_{k-1})_{ef} := b_e + c_f$ for all $e \in A, f \in \delta^+(e^-)$ and $(\hat{Q}_{k-1})_{ef} := 0$ otherwise, where $b, c \in \mathbb{R}^m$ are obtained from (\ref{equivalentrepresent}).
\end{remark}

\noindent Hu and Sotirov \cite{HuSotirov2} show that in the case that the linearizable matrix $\hat{Q}$ is in the form $\hat{Q} = [U^T, V^T]Y + \text{Diag}(z)$
for some $Y \in \mathbb{R}^{2n \times m}$ and $z \in \mathbb{R}^m$, the bound $v_{RBB}$ is dominated by the solution of the first level RLT relaxation  introduced by Adams and Sherali \cite{AdamsSherali1, AdamsSherali2}.
Here RLT stands for reformulation linearization technique.
 In \cite{HuSotirov2} it is moreover shown that the first level RLT  bound, denoted by    $v_{RLT1}$,
  can be obtained by searching for the optimal linearizable matrix $\hat{Q}$ of the form $\hat{Q} = [U^T, V^T]Y + M + \text{Diag}(z)$ where $M$ is a skew symmetric matrix.

Our preliminary numerical results show that the above algorithm does not improve significantly the $LBB1$ bound. However, in the next section we show that
our approach outperforms known iterative approaches related to the Gilmore-Lawler bounds.

\section{The Gilmore-Lawler Type Bound} \label{GilmoreLawler}
In this section we consider the classical Gilmore-Lawler type bound. The GL procedure is a well-known approach to construct lower bounds for quadratic 0-1 optimization problems, see e.g.\ \cite{Gilmore, Lawler, RostamiMalucelli, Rostami}. We provide a compact formulation of the GL type bound that can be used to compute the bound by a single LP-problem, instead of solving $m$ subproblems. Moreover, we show that this bound in fact belongs to the family of linearization based bounds. Therefore, based on the results of Section \ref{ReformulatedLBBApproach}, we provide a bounding procedure that computes the best GL type bound in each step of the algorithm. We conclude this section by testing this new bounding procedure on some preliminary test instances.

\subsection{The classical GL type bound} \label{classicalGL}
In the objective function of the \textsc{QCCP}, see (\ref{QCCPdefinition}), we have the quadratic term $x_ex_f$ for each two arcs $e,f \in A$ placed in succession on a cycle. To get rid of this quadratic term, for each given arc $e \in A$, potentially in the solution, we consider the cycle cover containing $e$ with minimum interaction cost with $e$. We denote this minimum contribution of arc $e$ to a solution by $z_e$.
In particular, for all $e \in A$ we have
\leqnomode
\begin{equation}
z_e := \min\{Q_{e,:}x \, | \, \, x \in X, x_e = 1\} \tag{$P_e$},
\end{equation}
where  $Q_{e,:}$ denotes the $e$th row of the cost matrix $Q$.
The feasible set of $(P_e)$ equals the set of all node-disjoint cycle covers containing arc $e$.
If this set is empty, then we set $z_e$ equal to 0 since  arc $e$ cannot contribute to a cycle cover. \\
Let $z \in \mathbb{R}^m$ be the vector consisting of the elements $z_e$ for all $e \in A$.
Then the classical GL type bound is obtained by solving the following  \textsc{CCP}:
\begin{align*}
v_{GL} := \min\{z^Tx \, | \, \, x \in X\} \tag{$GL$}.
\end{align*}
Note that the constraint matrices of $(P_e)$ and $(GL)$ are totally unimodular. For this reason, we can drop the integrality constraints and solve $(GL)$ and $(P_e)$ for all $e \in A$ as linear programming problems. \\ \\
Besides computing the GL type bound by solving $(GL)$ and $(P_e)$ for all $e \in A$, we can also obtain its value by solving an integer linear programming (ILP) problem.
The problem $(GL_{ILP})$ is defined as follows:
\reqnomode
\begin{flalign}
(GL_{ILP}) && \min \quad & \sum_{e \in A} \sum_{f \in A}Q_{ef}y_{ef}  \nonumber \\
&& \text{s.t.} \quad & \sum_{f \in \delta^+(i)} y_{ef} = \sum_{f \in \delta^-(i)} y_{ef} = x_e & & \forall i \in N, \forall e \in A & \label{glilp1} \\
&& & y_{ee} = x_e & & \forall e \in A & \label{glilp2} \\
&& & y_{ef} \in \{0,1\}, \, x \in X & & \forall e, f \in A  \label{glilp3}
\end{flalign}
It follows from the constraints that if $x_e = 1$, then $y_{e,:} := [y_{e1}, ..., y_{em}]$ is the characteristic vector of the cheapest cycle cover containing arc $e$ and if $x_e = 0$, then $y_{e,:}$ equals the zero vector.

Let $(CGL_{ILP})$ be the continuous relaxation of $(GL_{ILP})$. In this continuous relaxation we can omit the upper bounds on $x_e$ and $y_{ef}$ for all $e, f \in A$, since it is never optimal to set the value of these variables larger than one. We can compute the GL type bound by solving $(CGL_{ILP})$ as stated by the following theorem. This theorem is based on a similar result for the \textsc{QMST}, see \cite{RostamiMalucelli}.
\begin{theorem}
The optimal value of $(CGL_{ILP})$ equals $v_{GL}$. \label{CGLProof}
\end{theorem}
\begin{proof}
Let $\lambda_{e,i}$ and $\alpha_{e,i}$ be the dual variables corresponding to constraints (\ref{glilp1}), i.e.\ $\lambda_{e,i}$ corresponds to $\sum_{f \in \delta^+(i)}y_{ef} = x_e$ and $\alpha_{e,i}$ corresponds to $\sum_{f \in \delta^-(i)}y_{ef} = x_e$. Similarly, let $\mu_i$ and $\gamma_i$ be the dual variables corresponding to the first and second equalities of constraints (\ref{X}), and $\theta_e$ the dual variable corresponding to constraints (\ref{glilp2}). The dual problem of $(CGL_{ILP})$ is as follows:
\begin{flalign}
 (DCGL_{ILP})  \qquad \max \quad & \sum_{i \in N} \mu_i + \sum_{i \in N}\gamma_i \nonumber \\
 \text{s.t.} \quad & \lambda_{e,f^+} + \alpha_{e,f^-} \leq Q_{ef} && \forall e, f \in A , \, f \neq e \quad\label{dglilp1}  \\
 & \lambda_{e,e^+} + \alpha_{e,e^-} + \theta_e \leq Q_{ee} && \forall e \in A & \label{dglilp2}\\
 & - \sum_{i \in N}\lambda_{e,i} - \sum_{i \in N}\alpha_{e,i} + \gamma_{e^-} + \mu_{e^+}  - \theta_e \leq 0 & &  \forall e\in A.  \label{dglilp3}
\end{flalign}
Constraint (\ref{dglilp3}) can be rewritten as  $\gamma_{e^-} + \mu_{e^+}  \leq \sum_{i \in N}\lambda_{e,i} + \sum_{i \in N}\alpha_{e,i} + \theta_e$ for all $e \in A$. In order to maximize the objective function of $(DCGL_{ILP})$, we maximize the right hand side of this inequality subject to constraints (\ref{dglilp1})-(\ref{dglilp2}). This gives for each $e \in A$ the following subproblem:
\leqnomode
\begin{align*}
z'_e : = \max & \left\lbrace \sum_{i \in N}\lambda_{e,i} + \sum_{i \in N}\alpha_{e,i} + \theta_e \, | \, \, (\ref{dglilp1}),  (\ref{dglilp2}) \right\rbrace . \tag{$DCP_e$}
\end{align*}
\reqnomode
For each fixed $e \in A$ the subproblem given above equals the dual of the continuous relaxation of $(P_e)$. By strong duality we know $z'_e = z_e$ for all $e \in A$. Substitution of this term into constraint (\ref{dglilp3}) gives a rewritten formulation for $(DCGL_{ILP})$:
\begin{align*}
\max & \left\lbrace \sum_{i \in N} \mu_i + \sum_{i \in N}\gamma_i \, | \, \,   \mu_{e^+} + \gamma_{e^-} \leq z_e \quad \forall e \in A \right\rbrace .
\end{align*}
This problem exactly equals the dual of the continuous relaxation of $(GL)$. Because of strong duality, it follows that the optimal objective value of $(CGL_{ILP})$ equals $v_{GL}$.
\end{proof}

We can show that the Gilmore-Lawler type bound for the \textsc{QCCP} in fact belongs to the family of linearization based bounds introduced in Section \ref{LinearizationBasedBounds}. That is, we can obtain $v_{GL}$ by searching for a linearizable quadratic cost matrix $\hat{Q}$ of a specific type that is as close as possible to $Q$.
The required linearizability condition on  $\hat{Q}$ is given below, and it  differs  from the sufficient conditions  presented in Section \ref{SufficientConditions}.

\begin{proposition} \label{GLsufficientcondition}
If there exists $B, C \in \mathbb{R}^{m \times n}$ and $t \in \mathbb{R}^m$ such that $Q_{ef} = B_{e, f^+} + C_{e, f^-}$ for $e \neq f$ and $Q_{ee} = B_{e,e^+} + C_{e,e^-} + t_e$ for all $e \in A$, then $Q$ is linearizable with vector $p$ given by $p_e = t_e + \sum_{i = 1}^nB_{e,i} + \sum_{i = 1}^nC_{e,i}$.
\end{proposition}
\begin{proof}
Let $\tilde{Q}$ be defined as $\tilde{Q}_{ef} = B_{e,f^+} + C_{e,f^-}$ for all $e,f \in A$. Then $\tilde{Q}$ can be seen as a generalized weak sum matrix where $D$ and $T$ are equal to the zero matrix, see Definition \ref{generalizedweaksumdef}. According to Proposition \ref{generalizedweaksum}, $\tilde{Q}$ is linearizable with vector $\tilde{p} = \sum_{i =1}^nB_{e,i} + \sum_{i = 1}^nC_{e,i}$. Since $Q = \tilde{Q} + \text{Diag}(t)$, it follows that $Q$ is linearizable with vector $p$ given by $p_e = t_e + \sum_{i = 1}^nB_{e,i} + \sum_{i = 1}^nC_{e,i}$.
\end{proof}
Similar to the notation used in Section \ref{LinearizationBasedBounds}, let $S_{GL}(Q)$ denote the set of all linearizable cost matrices $\hat{Q} \in \mathbb{R}^{m \times m}$ that satisfy the conditions of Proposition \ref{GLsufficientcondition}. Moreover, let $\tau_{GL}(Q)$ be the following polyhedron:
\begin{align}\label{tauGLQ}
\tau_{GL}(Q) &:= \{\hat{p} \in \mathbb{R}^m \, | \, \, x^T\hat{p} = x^T\hat{Q}x \text{ for all } x \in X \, , \, \, \hat{Q} \in S_{GL}(Q)\},
\end{align}
and
\begin{align} \label{v^GL_LBB}
v^{GL}_{LBB} := \max_{\substack{y \in \mathbb{R}^{2n} \\ \hat{p} \in \mathbb{R}^m}} \{ \mathbbm{1}_{2n}^Ty \, | \, [U^T, V^T]y \leq \hat{p}, \, \, \hat{p} \in \tau_{GL}(Q)\}.
\end{align}
Now we prove the main result of this section which states that the classical Gilmore-Lawler type bound can be seen as a special case of linearization based bound.
\begin{theorem} \label{EquivalenceGLLBB}
We have $v_{LBB}^{GL} = v_{GL}$.
\end{theorem}
\begin{proof}
By using the polyhedral description of $S_{GL}(Q)$ following from Proposition \ref{GLsufficientcondition}, the optimization problem in (\ref{v^GL_LBB}) can be written as follows:
\begin{align}
v_{LBB}^{GL} =\max \quad & \sum_{i = 1}^{2n}y_i \\
\text{s.t.} \quad & [U^T, V^T]y \leq \hat{p} \label{LBB_GL0}\\
& B_{e,f^+} + C_{e,f^-} \leq Q_{ef} \quad \qquad \quad \, \, \, \, \, \forall e,f \in A, f \neq e \label{LBB_GL1} \\
& B_{e,e^+} + C_{e,e^-} + t_e \leq Q_{ee} \quad\, \qquad \forall e \in A \label{LBB_GL2}\\
& \hat{p}_e = t_e + \sum_{i = 1}^nB_{e,i} + \sum_{i = 1}^nC_{e,i} \qquad \forall e \in A \label{LBB_GL4}\\
& y \in \mathbb{R}^{2n}, \hat{p}, t \in \mathbb{R}^m, B, C \in \mathbb{R}^{m \times n}. \label{LBB_GL5}
\end{align}
We show that this optimization problem is equivalent to $(DCGL_{ILP})$, the dual problem of the continuous relaxation of $(GL_{ILP})$. Take $B_{e,i} = \lambda_{e,i}$ and $C_{e,i} = \alpha_{e,i}$ for all $e \in A$ and $i \in N$, where $\lambda$ and $\alpha$ denote the dual vectors belonging to constraints (\ref{glilp1}). Similarly, let $t = \theta$ where $\theta$ equals the dual vector to constraints (\ref{glilp2}). Finally, let $y = [\mu^T, \gamma^T]^T$, where $\mu$ and $\gamma$ are the dual variables belonging to constraints (\ref{X}). By substitution of these variables and combining constraints (\ref{LBB_GL0}) and (\ref{LBB_GL4}), we obtain the problem $(DCGL_{ILP})$, i.e., the dual of $(CGL_{ILP})$. Thus, we have $v_{LBB}^{GL} = v_{GL}$.
\end{proof}
Theorem \ref{EquivalenceGLLBB} shows that the GL type bound belongs to the family of linearization based bounds. This is also shown by Hu and Sotirov \cite{HuSotirov2} and Rostami et al.\ \cite{Rostami}, however our proof is very different as we exploit the fact that $v_{GL}$ can be obtained by solving an LP problem, i.e., $(CGL_{ILP})$.
Additionally, we show here that the computation of the GL type bound is equivalent to the search for the optimal linearizable cost matrix $\hat{Q}$ satisfying the properties of Proposition \ref{GLsufficientcondition}.

\subsection{The best Gilmore-Lawler type bound} \label{bestGL}
Section \ref{classicalGL} shows that the calculation of the classical GL type bound fits in the general framework discussed in Section \ref{LinearizationBasedBounds}. In this section we apply the reformulation procedure of Section \ref{ReformulatedLBBApproach} to the GL type bound. We also show that our approach outperforms several iterative approaches from the literature.

In order to apply Algorithm \ref{AlgorithmLBB} to the sufficient condition for linearizability of Proposition \ref{GLsufficientcondition}, we need to define how to calculate $r_k$ for each iteration $k$, see (\ref{equivalentrepresent}).
We rewrite the set $\tau_{GL}(Q)$, see  \eqref{tauGLQ}, as follows:
\begin{align*}
\tau_{GL}(Q) = \{\hat{p} \in \mathbb{R}^m \, | \, \, t \in \mathbb{R}^m, B, C \in \mathbb{R}^{m \times n}, (\ref{LBB_GL1}), (\ref{LBB_GL2}), (\ref{LBB_GL4}) \},
\end{align*}
which is clearly a polyhedron. Then for all $k \geq 1$ we calculate the additional amount of quadratic cost that is linearized by solving:
\begin{align}
r_k := \max_{\substack{y \in \mathbb{R}^{2n} \\ p_k \in \mathbb{R}^m \\ \eta \in [0,1]}} \{ \mathbbm{1}_{2n}^T y \, | \, \, [U^T, V^T ] y = p_k , p_k \in \tau_{GL}(Q_{k-1}^\eta) \}. \label{GLr_k}
\end{align}
Note that as opposed to the constraints in (\ref{equivalentrepresent}), we have replaced the constraint $[U^T, V^T]y \leq p_k$ by an equality constraint. This does not change the value of $r_k$. To verify this, suppose we solve (\ref{GLr_k}) using the inequality constraint $[U^T, V^T]y \leq p_k$ and let $\hat{y}, \hat{p}_k$ and $\hat{t}$ be the corresponding optimal solutions. Let $e \in A$ be such that the inequality constraint is satisfied with strict inequality. Then without changing $\hat{y}$, we can reduce $\hat{t}_e$ (and thus $\hat{p}_e$) such that we get equality for $e \in A$. Although it changes the linearization vector $\hat{p}$, the resulting bound remains equal. To verify this, notice that only the left hand side of constraint (\ref{LBB_GL2}) is decreased, so the solution is still feasible and the optimal value $r_k$ remains unchanged. From this, it follows that one may replace $[U^T, V^T]y \leq p_k$ by an equality constraint and solve $r_k$ as in (\ref{GLr_k}).

Algorithm \ref{AlgorithmLBB} using (\ref{GLr_k}) to compute $r_k, p_k$ and $\eta$ gives a new bounding procedure for the \textsc{QCCP}. We call the resulting bound the reformulated GL type bound ($RGL$) and denote its value by $v_{RGL}$. By construction, it iteratively computes the best Gilmore-Lawler type bound among all equivalent representations of the quadratic cost matrix.

The algorithm  proposed in this section satisfies another interesting property, namely the vectors $d_k$ satisfy the constant value property for all $k \geq 0$.
This is an important property for linearizability because the set of linearizable cost matrices for combinatorial optimization problems with interaction costs can be
characterized by the constant value property, under certain conditions, see \cite{LendlCustiPunnen}.

\begin{theorem} \label{CVP}
All $d_k$ where $k \geq 0$ computed during the $RGL$ approach, satisfy the constant value property, i.e., we have $x^Td_k = \bar{x}^Td_k$ for all feasible cycle covers $x, \bar{x} \in X$.
\end{theorem}
\begin{proof}
We apply a proof by induction on $k$. Note that the vector $d_0$ equals the $m \times 1$ vector of zeros which trivially satisfies the constant value property.

Now assume that the induction hypothesis is true for iteration $k-1$, i.e., $x^Td_{k-1} = \bar{x}^Td_{k-1}$ for all feasible cycle covers $x, \bar{x} \in X$. In iteration $k$ we solve (\ref{GLr_k}). Let $\hat{y} \in \mathbb{R}^{2n}$ and $\hat{p}\in \mathbb{R}^m$ be the optimal variables for this problem and split $\hat{y} = [\mu^T, \lambda^T]^T$ with $\mu, \lambda \in \mathbb{R}^n$. It follows that $[U^T, V^T]\hat{y}= U^T\mu + V^T\lambda = \hat{p}$. Now let $x \in X$ be any feasible cycle cover in $G$. Then we can sum up the rows of this system of equalities for all arcs $e \in A$ in the cycle cover implied by $x$:
\begin{align*}
\sum_{\{e \in A: x_e = 1\}} (\mu_{e^+} + \lambda_{e^-})  = \sum_{\{e \in A: x_e = 1\}}\hat{p}_e \quad \Leftrightarrow \quad \sum_{i \in N}\mu_i + \sum_{i \in N} \lambda_i  = x^T\hat{p}
\end{align*}
where we use the fact that each vertex is visited exactly once on a cycle cover. So the quantity $x^T\hat{p}$ is equal for all $x \in X$. As a result, $\hat{p}$ satisfies the constant value property. \\
The vector $d_k$ is constructed as $d_{k-1} + p_k$ with $p_k = \hat{p}$. Since $d_{k-1}$ and $\hat{p}$ satisfy the constant value property, it follows that $d_k$ satisfies the constant value property.
\end{proof}
\begin{remark} Since the GL type bound can be computed both as a linearization based bound and by solving $(CGL_{ILP})$ (see Theorem \ref{EquivalenceGLLBB}), the iterative approach derived in this section can also be defined in terms of $(CGL_{ILP})$. In that case, we iteratively compute $v_{GL}$ and reformulate the quadratic cost matrix using the dual variables of $(CGL_{ILP})$. The details of this equivalent approach can be found in \cite{Meijer}.
\end{remark}

\noindent Since the linearizable matrix $\hat{Q}$ of Proposition \ref{GLsufficientcondition} can be written in the form $\hat{Q} = [U^T, V^T]Y + \text{Diag}(z)$ for some $Y \in \mathbb{R}^{2n \times m}$ and $z \in \mathbb{R}^m$, it follows from \cite{HuSotirov2} that we have $v_{RGL} \leq v_{RLT1}$.

\subsection{Preliminary Results} \label{Prelim2}
For the instances considered in the preliminary results of Section \ref{LinearizationBasedBounds}, we now test our Gilmore-Lawler type bounds. First, we compute the classical GL type bound, after symmetrizing the quadratic cost matrix $Q$. This bound is denoted by $GL$. Moreover, we consider the iterative GL type bounding approach where we symmetrize the quadratic cost matrix in each iteration. That is, we apply Algorithm \ref{AlgorithmLBB} using (\ref{GLr_k}) where instead of optimizing over $\eta$, we set $\eta = \frac{1}{2}$. We denote this bound by $RGL^{sym}$. Finally, we report the bound $RGL$ which is introduced in Section \ref{bestGL}. The maximum computation time is set at 3600 seconds. The results are given in Table \ref{GLPreliminary}.
\begin{table}[H]
\footnotesize
\centering
\begin{tabular}{@{}cccccccccc@{}}
\toprule
    &     &   &    & \multicolumn{2}{c}{$GL$} & \multicolumn{2}{c}{$RGL^{sym}$} & \multicolumn{2}{c}{$RGL$}  \\ \cmidrule(l){5-6} \cmidrule(l){7-8} \cmidrule(l){9-10}
$p$ & $n$ & $m$ &  $OPT$ & bound          & time & bound          & time              & bound & time \\ \midrule
0.1 &  20 &44 & 923 & 923 & 0.017  & 923  & 0.015  & 923  & 0.088  \\
 & 25& 76 & 1039 & 681 & 0.039 & 864 & 4.652  & 1018& 95.52   \\
  &30 & 100 & 1082 & 781 & 0.053   & 899 & 2.412   & 1082 & 15.18  \\
0.3 & 15 & 61 & 485   &      347     &   0.061   & 368& 1.481     & 485  & 7.140           \\
    & 20 & 118 & 438   & 223          & 0.068    & 263& 3.482       & 418  & 3600            \\
    & 25 &  172 & 382   & 176          & 0.136   &190 &       5.105 &         276 &    3600       \\
0.5 & 15 &  116 & 226   & 102          & 0.336   &110 &       2.602 &          222 &    1835        \\
    & 20 & 177 &    255   &    93            &     0.118                  & 103& 5.365 &  169	 & 3600               \\
    & 25 & 306 &   n.a.   &    66   &   0.296 & 75 &  13.80  &               105 & 3600               \\
0.7 & 15 & 149 & 173& 63 & 0.080  & 67 & 3.530  & 117  & 1236 \\
& 20 & 263& n.a. & 52  & 0.181  & 54 & 9.624  & 63  &  269.3 \\
 & 25& 396& n.a. & 56  & 0.365  & 62 & 18.77  & 79 &   1085 \\
    \bottomrule
\end{tabular}
\caption{Bounds and computation times in seconds of GL type bounds on Erd\H os-R\' enyi instances. \label{GLPreliminary}}
\end{table}
\noindent From Table \ref{GLPreliminary} it follows that the iterative approaches significantly improve the classical GL type bound. Among these iterative approaches, $RGL$ provides much stronger bounds than $RGL^{sym}$.
We conclude that this new approach of calculating the best GL type bound in each step provides better bounds than when setting $\eta = \frac{1}{2}$ in the reformulation.
However, it turns out that this improvement in the quality comes at the cost of efficiency. Clearly, we can stop our algorithm after a pre-specified number of iterations and/or time.

\section{Other bounds for the \textsc{QCCP}} \label{OtherBounds}
In this section we present several known bounding approaches from the literature that can be applied to the \textsc{QCCP}. In the next section, we compare those bounds with the bounds introduced in this paper.
We consider a column generation approach and a mixed integer linear programming (MILP) based  bound. \\ \\
Galbiati et al.\ \cite{Galbiati} construct a column generation approach for the \textsc{MinRC3}. This approach can be extended to the \textsc{QCCP}. To the best of our knowledge, this is the only implemented lower bounding approach for the \textsc{MinRC3} in the literature.

Let $\mathcal{C}$ be the set of all directed cycles in $G$. Moreover, let $\overline{\mathcal{C}} \subseteq \mathcal{C}$ be a subset of cycles such that it contains at least one cycle cover. Let $w_c$ be the cost of a cycle $c \in \mathcal{C}$. Then the restricted master problem $(RMP)$ is given by:
\begin{flalign}
(RMP) && \min_y \quad & \sum_{c \in \overline{\mathcal{C}}}w_cy_c \nonumber \\
&& \text{s.t.} \quad &  \sum_{c \in \overline{\mathcal{C}}: i \in c} y_c = 1 & & \forall i \in N & \label{RMP1} \\
&& & y_c \geq 0 & & \forall c \in \overline{\mathcal{C}}. & \label{RMP2}
\end{flalign}

Let $\pi \in \mathbb{R}^n$ be the vector of dual variables corresponding to constraint (\ref{RMP1}). Then the subproblem $(SP)$ searches for the cycle in $\mathcal{C}$ with the smallest (negative) reduced costs with respect to $\pi$, i.e.
\begin{flalign*}
(SP) && \min_{x,z} \{ x^TQx - z^T\pi \,  | \, \, \sum_{e \in \delta^+(i)} x_e = \sum_{e \in \delta^-(i)}x_e = z_i \quad \forall i \in N, \quad &  &  \\
&& \qquad  \, \sum_{e \in A}x_e \geq 2, \, x \in \{0,1\}^m, \, z \in \{0,1\}^n\}, & &
\end{flalign*}
where $z_i = 1$ if vertex $i$ is on the cycle and 0 otherwise. As stated in \cite{Galbiati}, the problem $(SP)$ is strongly $\mathcal{NP}$-hard itself. The quadratic objective function can be linearized by standard linearization techniques. A lower bound on the optimal value of the \textsc{QCCP} can be obtained by iteratively solving the master problem and its corresponding subproblem. If a cycle with negative reduced cost is found, we add it to the set $\overline{\mathcal{C}}$. This procedure is repeated until no more cycle with negative reduced cost is found or after some predefined stopping criteria. The obtained bound is denoted by $v_{CG}$.\\

Based on a procedure by \cite{Glover, Adams}, we  present the \textsc{QCCP} as an MILP problem.
{Let us first fix an  equivalent representation of $(G,Q)$.
Let $z_e$ be computed as in $(P_e)$ for all $e \in A$, see Section \ref{classicalGL}.
Moreover, we define for all $e \in A$
\begin{align*}
q_e^{\max} := \max \{ Q_{e,:}x \, : \, \, x \in X, x_e = 0 \}.
\end{align*}
Note that $q_e^{\max}$ can be obtained by solving a linear programming problem.} Then the \textsc{QCCP} can be formulated as an MILP:
\begin{flalign}
   (MILP) && \min_{x,y}  \quad & \sum_{e \in A}y_e  \nonumber \\
    && \text{s.t.} \quad & y_e \geq z_ex_e &&   \forall e \in A & \label{milp1} \\
    && & y_e \geq Q_{e,:}x - q_e^{\max} (1 - x_e) &&  \forall e \in A & \label{milp2} \\
    && & x \in X, y \in \mathbb{R}^m . \nonumber %%\label{milp3}
\end{flalign}
If we relax the binary constraint on $x$, then we obtain a lower bound for the \textsc{QCCP}. We call this bound the MILP based bound and we denote its value by $v_{MILP}$.
{The next result shows that $v_{MILP}$ is at least as large as the Gilmore-Lawler type bound.
\begin{theorem}
The MILP based bound dominates the Gilmore-Lawler type bound, i.e., $v_{GL} \leq v_{MILP}$.
\end{theorem}
\begin{proof}
Let $\beta, \delta \in \mathbb{R}^m_+$ denote the dual variables of (\ref{milp1}) and (\ref{milp2}), respectively. Moreover, let $\mu, \gamma \in \mathbb{R}^n$ denote the dual variables of the cycle cover constraints $\sum_{e \in \delta^+(i)}x_e = 1$ and $\sum_{e \in \delta^-(i)}x_e = 1$ for all $i \in N$, respectively. Then the dual of the MILP based bound equals
\begin{flalign*}
(DMILP) && v_{MILP} := \max_{\beta, \delta, \mu, \gamma} \quad & \sum_{i \in N} \mu_i + \sum_{i \in N} \gamma_i - \sum_{e \in A} \delta_e q_e^\text{max}  \\
&& \text{s.t.} \quad & \beta_e + \delta_e = 1 & & \forall e \in A &  \\
&& & \mu_{e^+} + \gamma_{e^-} \leq z_e\beta_e + \delta^T Q_{:,e} + \delta_e q_e^{\max} & & \forall e \in A  & \\
&& & \beta_e, \delta_e \geq 0 & & \forall e \in A,
\intertext{where $Q_{:,e}$ equals the $e$th column of $Q$. Now set $\delta_e = 0$ for all $e \in A$. Then, $\beta_e = 1$ for all $e \in A$ due to the first set of constraints above. Then, $(DMILP)$ reduces to}
& &  \max_{\mu, \gamma} \quad & \sum_{i \in N} \mu_i + \sum_{i \in N} \gamma_i  \\
& &  \text{s.t.} \quad & \mu_{e^+} + \gamma_{e^-} \leq z_e & & \forall e \in A.
\end{flalign*}
This problem equals the dual of the continuous relaxation of $(GL)$.
%%Let $\mu^* \in \mathbb{R}^n$ and $\gamma^* \in \mathbb{R}^n$ be optimal for the latter problem. %Observe that the problem above equals the dual of the continuous relaxation of $(GL)$.
%%Hence, by taking $\delta$ and $\beta$ as the vectors of zeros and ones, respectively, $\mu = \mu^*$ and $\gamma = \gamma^*$, we obtain a feasible solution for $(DMILP)$ with objective value $v_{GL}$.
Hence, it follows that $v_{GL} \leq v_{MILP}$.
\end{proof}
%%Note that both,  the MILP based bound and the Gilmore-Lawler type bound, depend on the equivalent reformulation of $(G,Q)$
Note the MILP based bound and the Gilmore-Lawler type bound are comparable if the same equivalent reformulation of $(G,Q)$ is used in their computations.
%%under the assumptions that one uses the same equivalent reformulation of $(G,Q)$
}

\section{Numerical Results} \label{NumericalResults}

In this section we test our bounding approaches on a set of test instances and compare them with several approaches from the literature. We take into account the linearization based bound $LBB1$ from Section \ref{LBBConstruction}, the classical GL type bound $GL$ from Section \ref{classicalGL}, the reformulated GL type bound $RGL$ discussed in Section \ref{bestGL},  the column generation approach $CG$ and the MILP based bound $MILP$ from Section \ref{OtherBounds}, and  the first level RLT bound $RLT1$, see \cite{AdamsSherali1, AdamsSherali2}.  The GL bound and the MILP based bound are computed after symmetrizing $Q$. Note that we do not take into account $LBB2$ and $LBB3$, since our preliminary experiments from Section \ref{Prelim1} show that $LBB1$ is preferred when taking both quality and efficiency into account. \\
All lower bounds are implemented in Matlab on a pc with an Intel(R) Core(TM) i5-6500 CPU, 3.20 GHz and 8 GB memory using CPLEX 12.6 as the solver. \\ \\
We consider the following types of instances:
\begin{itemize}
\item \textbf{Erd\H os-R\'enyi instances:} These instances are created via the $G(n,p)$ Erd\H os-R\' enyi model \cite{ErdosRenyi}. The number of nodes is fixed to $n$ and each arc is included with probability $p$ independent of the other arcs. The quadratic cost between any pair of successive arcs is chosen discrete uniformly at random out of $\{0, ..., 100\}$.
\item \textbf{Manhattan instances:} The Manhattan instances are introduced in \cite{Comellas} and resemble the street pattern of modern cities like Manhattan or Barcelona. Given a finite set of positive integers $(n_1, n_2, ..., n_k)$, the graph consists of a $n_1 \times n_2 \times ... \times n_k$ directed grid. Each node in the interior is adjacent to its {$2k$ neighbours}. The nodes on the boundary are also incident to the corresponding nodes on the opposite boundary. For each dimension $k$, the arcs belonging to the same layer of the grid point in the same direction. However, the arcs of two consecutive layers point in the opposite direction. This results in a graph containing a large number of cycles. The quadratic cost between any pair of successive arcs is chosen discrete uniformly at random out of $\{0, ..., 10\}$.
\item \textbf{Angle-distance instances:} The Angle-distance instances are originally constructed for the \textsc{QTSP} in \cite{Fischer}. The number of nodes $n$ and the graph density $p$ are given. The $(x,y)$-coordinates of each node is chosen discrete uniformly at random out of $\{0, ..., 500\}^2$. Exactly $\lceil pn(n-1) \rceil$ arcs are chosen uniformly at random from the total set of arcs. For each arc $e \in A$, let $d_e$ denote the Euclidean distance between the endpoints of $e$. Moreover, for each two successive arcs $e$ and $f$, let $\alpha_{ef}$ denote the turning angle (in radians). Given some constant $\rho \in \mathbb{R}_+$, the quadratic cost of two successive arcs $e$ and $f$ is calculated as:
\begin{align*}
Q_{ef} = \left\lceil 0.1 \left( \rho \cdot \alpha_{ef} + \frac{d_e + d_f}{2} \right) \right\rceil.
\end{align*}
Similar as in \cite{Fischer}, we take $\rho = 40$.
\end{itemize}
\noindent For Erd\H os-R\' enyi and Angle-distance instances, preliminary experiments show that instances up to approximately 300 arcs can be solved to optimality within one hour.
For the Manhattan instances the limit is around 2000 arcs, due to the small density of these types of graphs.

In total we consider two sets of experiments: experiments on small instances and experiments on large instances. Since the optimum, $RLT1$ and $CG$ cannot be calculated for large instances, we only test these approaches on the smaller instances. Moreover, we include the bounds introduced in this paper, namely $LBB1$ and $RGL$. The value and computation times (in seconds)  on small Erd\H os-R\'enyi instances can be found in Table \ref{SmallER}. This table contains 6 instances for $n = 20, 25, 30$ and $p = 0.3, 0.5$. The results on  Manhattan and Angle-distance instances are reported in Tables \ref{SmallMH} and \ref{SmallAD}, respectively. For the Angle-distance instances we take the same values for $n$ and $p$ as for the Erd\H os-R\'enyi instances, while for the Manhattan instances we consider several two- and three-dimensional instances. The maximum computation time is set to 3600 seconds. When after this time no bound is computed, we report `n.a.' in the tables. Since the optimal value is always integer, we round up all bounds.

For the smaller instances, we see that $RLT1$ performs best in quality. When it can be computed, it is often close to the optimal value and it dominates the other bounds. $LBB1$ is often very close to $RLT1$, but can be computed much more efficiently. Namely, for the Erd\H os-R\' enyi and the Angle-distance instances the computation time of $LBB1$ for all small instances is below 0.4 seconds, whereas $RLT1$ cannot be computed within one hour for some of these instances. The column generation approach provides strong bounds, but in most cases it is not able to compute a bound in a time span of one hour. The reformulated GL type bound performs well on the Manhattan and Angle-distance instances, see Tables \ref{SmallMH} and \ref{SmallAD}. Although its total computation time is large, the advantage of this approach is that it provides a bound in a short time and then iteratively  improves the value. This makes it possible to stop the procedure at any given time and still obtain a bound. The bounds computed by $RGL$ are in almost all cases dominated by $LBB1$.

When taking both efficiency and quality into account, we conclude that the linearization based bound $LBB1$ outperforms the other approaches. Based on Tables \ref{SmallER}, \ref{SmallMH} and \ref{SmallAD}, the value of $LBB1$ is at least 75\% of the optimal value for the Erd\H os-R\'enyi instances. For the Angle-distance and Manhattan instances, this percentage equals 98\% and 96\%, respectively. \\ \\

\begin{table}[H]
\scriptsize
\centering
\begin{tabular}{@{}ccccccccccccc@{}}
\toprule
 & & & \multicolumn{2}{c}{$OPT$} & \multicolumn{2}{c}{$RLT1$} & \multicolumn{2}{c}{$CG$} & \multicolumn{2}{c}{$LBB1$} & \multicolumn{2}{c}{$RGL$} \\
 \cmidrule(l){4-5} \cmidrule(l){6-7} \cmidrule(l){8-9} \cmidrule(l){10-11} \cmidrule(l){12-13}
 $p$ & $n$ & $m$ & value & time & value & time & value & time & value & time & value & time \\
 \midrule
 0.3 & 20 & 119 & 319 & 10.28 & 301 & 4.825 & 289 & 102.3 & 260 & 0.020 & 285 & 3600\\
  & 25 & 177 & 386 & 19.04 & 331 & 20.09 & 331 & 928.9 & 305 & 0.040 & 280 & 3600\\
  & 30 & 280 & n.a. & 3600 & 284 & 70.62 & n.a.& 3600 & 274 & 0.134 & 185 & 3600\\
  0.5 & 20 & 195 & 236 & 211.0 & 181 & 17.00 & n.a. & 3600 & 175 & 0.121 & 129 & 3600 \\
  & 25 & 327 & n.a. & 3600 & 141 & 82.52 & n.a. & 3600 & 136 & 0.233 & 89 & 3600 \\
  & 30 & 442 & n.a. & 3600 & 168 & 385.0 & n.a. & 3600 & 162 & 0.322 & 99 & 3600\\
  \bottomrule
\end{tabular}
\caption{Bounds and computation times in seconds of $RLT1$, $CG$, $LBB1$ and $RGL$ on small Erd\H os-R\' enyi instances. \label{SmallER}}
\end{table}

\begin{table}[H]
\scriptsize
\centering
\begin{tabular}{@{}ccccccccccccc@{}}
\toprule
 & & & \multicolumn{2}{c}{$OPT$} & \multicolumn{2}{c}{$RLT1$} & \multicolumn{2}{c}{$CG$} & \multicolumn{2}{c}{$LBB1$} & \multicolumn{2}{c}{$RGL$} \\
 \cmidrule(l){4-5} \cmidrule(l){6-7} \cmidrule(l){8-9} \cmidrule(l){10-11} \cmidrule(l){12-13}
 Instance & $n$ & $m$ & value & time  & value & time & value & time & value & time & value & time \\ \midrule
 $(5,5)$ & 25 & 50 & 103 & 0.483 & 103 & 1.534 & 103 & 1.484 & 103 & 0.006 & 103 & 5.756\\
 $(10,10)$ & 100 & 200 & 418 & 2.335 & 418 & 1.974 & 418 & 1645 & 418 & 0.022 & 371 & 16.06 \\
 $(4,4,4)$ & 64 & 192 & 199 & 6.312 & 193 & 9.415 & 196 & 691.4 & 193 & 0.081 & 175 & 3600\\
 $(6,6,6)$ & 216 & 648 &  700 & 23.67 & 683 & 1152 & n.a. & 3600 & 683 & 0.568 & 551 & 3600\\
 $(8,8,8)$ & 512 & 1536 & 1566 & 394.1 & n.a.& 3600& n.a. & 3600 & 1530 & 1.213 & n.a. & 3600 \\
 $(10, 10, 10)$ & 1000 & 3000 & n.a. & 3600 & n.a. &3600 & n.a.& 3600 & 3113 & 3.754  & n.a. & 3600 \\ \bottomrule
 \end{tabular}
 \caption{Bounds and computation times in seconds of $RLT1$, $CG$, $LBB1$ and $RGL$ on small Manhattan instances. \label{SmallMH}}
\end{table}

\begin{table}[H]
\scriptsize
\centering
\begin{tabular}{@{}ccccccccccccc@{}}
\toprule
 & & & \multicolumn{2}{c}{$OPT$} & \multicolumn{2}{c}{$RLT1$} & \multicolumn{2}{c}{$CG$} & \multicolumn{2}{c}{$LBB1$} & \multicolumn{2}{c}{$RGL$} \\
 \cmidrule(l){4-5} \cmidrule(l){6-7} \cmidrule(l){8-9} \cmidrule(l){10-11} \cmidrule(l){12-13}
 $p$ & $n$ & $m$ & value & time & value & time & value & time & value & time & value & time \\ \midrule
0.3 & 20 & 114 & 474 & 2.490 & 474 & 1.719 & 474 & 416.5 & 474 & 0.002 & 474 & 64.35\\
& 25 & 180 & 553 & 323.9 & 553 & 4.119 & n.a. & 3600 & 552 & 0.004 & 553 & 1559 \\
& 30 & 261 & 512 & 2951.0 & 512 & 19.52 & n.a. & 3600 & 512 & 0.079 & 494 & 3600\\
0.5 & 20 & 190 & 276 & 177.8 & 276 & 6.732 & n.a. & 3600 & 276 & 0.053 & 274 & 1319 \\
& 25 &300 & 342 & 2163.6 & 340 & 53.43 & n.a. & 3600 & 338 & 0.142 & 320 & 3600\\
& 30 & 435& n.a. & 3600 & 381 & 490.5 & n.a. & 3600 & 377 & 0.332 & 355 & 3600 \\ \bottomrule
 \end{tabular}
 \caption{Bounds and computation times in seconds of $RLT1$, $CG$, $LBB1$ and $RGL$ on small Angle-distance instances. \label{SmallAD}}
\end{table}

For the larger instances, we only compute the bounds that can be computed efficiently. That is, we do not consider the iterative approaches, but only the bounds $GL$, $MILP$ and $LBB1$.
We also investigate the effect of a reformulation by adding an optimal incident skew symmetric matrix to the cost matrix, see Section \ref{LBBConstruction}. We apply this reformulation to $LBB1$, which implies that we optimize over the set $\tau_1^{skew}(Q)$, see \eqref{skew}, instead of $\tau_1(Q)$. The resulting bound is denoted by $LBB1^{skew}$. For the Manhattan instances this bound is omitted, since preliminary experiments showed that this reformulation does not improve the bounds for most Manhattan instances. This could be due to the sparsity of Manhattan instances.
The bounds and computation times (in seconds) for the Erd\H os-R\'enyi, Manhattan and Angle-distance instances are reported in Tables
\ref{BigER}, \ref{BigMH} and \ref{BigAD}, respectively. For the Erd\H os-R\'enyi and Angle-distance instances we take for $n$ values between 30 and 100 nodes and consider $p = 0.3$ and $p = 0.5$. For the Manhattan instances we consider large two-dimensional instances and one large three-dimensional instance. The maximum computation time for these bounds is set to 1800 seconds. Again, we round up all bound values.

For the larger instances, we see that $LBB1$ in all cases dominates $GL$ and $MILP$ in both quality and efficiency. The difference in quality is most present for the Erd\H os-R\'enyi instances, see Table \ref{BigER}. For the Manhattan instances, we see that $GL$ and $MILP$ can be calculated efficiently for instances up to 3000 arcs.
 However, $LBB1$ remains efficient even for  larger instances. In particular, bounds for Manhattan instances up to 15000 arcs can be computed within  60 seconds.

Moreover, we conclude from Tables \ref{BigER} and \ref{BigAD} that the addition of an incidence skew symmetric matrix to the set $\tau_1(Q)$ only improves the bounds for some of the instances. In general, it turns out that the Erd\H os-R\'enyi instances can successfully be improved by this method, whereas for the Angle-distance instances only in a few cases there is an improvement.  Although the computation times of $LBB1^{skew}$ are  larger than those of $LBB1$, bounds can still be computed in a reasonable time span.

\begin{table}[H]
\scriptsize
\centering
\begin{tabular}{@{}ccccccccccccc@{}}
\toprule
 & & & \multicolumn{2}{c}{$GL$} & \multicolumn{2}{c}{$MILP$} & \multicolumn{2}{c}{$LBB1$}  & \multicolumn{2}{c}{$LBB1^{skew}$} \\
  \cmidrule(l){4-5} \cmidrule(l){6-7} \cmidrule(l){8-9}  \cmidrule(l){10-11}
  $p$ & $n$ & $m$ & value & time & value & time  & value & time & value & time \\ \midrule
  0.3 & 30 & 284 & 111 & 0.272 & 122 & 0.435 & 230 & 0.083 & 232 & 0.766 \\
  & 40 & 468 & 117 & 0.645 & 131 & 1.006 & 265 & 0.179 & 278 & 1.711 \\
  & 50 & 754 & 121 & 1.598 & 130 & 2.410 & 267 & 0.404 & 274 & 4.184 \\
  & 60 & 1062 & 103 & 4.068 & 118 & 5.788 & 272 & 0.726 & 272 & 8.048 \\
  & 70 & 1481 & 114 & 8.910 & 123 & 12.94 & 255 & 1.660 & 258 & 15.38\\
  & 80 & 1842 & 113 & 14.26 & 122 & 20.82 & 263 & 2.740 & 267 & 24.67 \\
  & 90 & 2385 & 114 & 23.25 & 122 & 37.61 & 259 & 5.296 & 261 & 41.74\\
  & 100 & 2962 & 119 & 36.03 & 126 & 63.49 & 269 & 13.32 & 270 & 69.90 \\
  0.5 & 30 & 434 & 73 & 0.557 & 79 & 0.783 & 161 & 0.182 & 163 & 9.218 \\
  & 40 & 793 & 69 & 1.607 & 74 & 2.364 & 166 & 0.554  & 169 & 10.38 \\
  & 50 & 1197 & 72 & 4.323 & 77 & 6.682 & 165 & 1.185 & 167 & 15.74 \\ \bottomrule
\end{tabular}
 \caption{Bounds and computation times in seconds of $GL$, $MILP$, $LBB1$ and $LBB1^{skew}$ on large Erd\H os-R\' enyi instances. \label{BigER}}
\end{table}

\begin{table}[H]
\scriptsize
\centering
\begin{tabular}{@{}ccccccccc@{}}
\toprule
 & & & \multicolumn{2}{c}{$GL$} & \multicolumn{2}{c}{$MILP$} & \multicolumn{2}{c}{$LBB1$}   \\
  \cmidrule(l){4-5} \cmidrule(l){6-7} \cmidrule(l){8-9}
    Instance & $n$ & $m$ & value & time  & value & time  & value & time  \\ \midrule
  $(20,20)$ & 400 & 800 & 1237 & 5.31 & 1472 & 7.491 & 1537 & 0.100  \\
  $(30,30)$ & 900 & 1800 & 2813 & 56.24 & 3343 & 86.46 & 3517 & 0.410  \\
  $(40,40)$ & 1600 & 3200 & 5101 & 346.8 & 6028 & 553.5 & 6302 & 1.388  \\
  $(50,50)$ & 2500 & 5000 & 7983 & 1225.3 & 9424 & 1897.8 & 9828 & 2.838  \\
  $(17,17,17)$ & 4913  & 14739 & n.a. & 1800 & n.a. & n.a. &  15398 & 54.79 \\ \bottomrule
\end{tabular}
 \caption{Bounds and computation times in seconds of $GL$, $MILP$ and $LBB1$ on large Manhattan instances. \label{BigMH}}
\end{table}

\begin{table}[H]
\scriptsize
\centering
\begin{tabular}{@{}ccccccccccccc@{}}
\toprule
 & & & \multicolumn{2}{c}{$GL$} & \multicolumn{2}{c}{$MILP$} & \multicolumn{2}{c}{$LBB1$} & \multicolumn{2}{c}{$LBB1^{skew}$}  \\
  \cmidrule(l){4-5} \cmidrule(l){6-7} \cmidrule(l){8-9} \cmidrule(l){10-11}
  $p$ & $n$ & $m$ & value & time & value & time & value & time & value & time \\ \midrule
  0.3 & 30 & 261 & 456 & 0.238 & 467 & 0.410 & 525 & 0.054 & 525 & 6.957\\
  & 40 & 468 & 507 & 0.693 & 516 & 0.984 & 567 & 0.463 & 567 & 7.795\\
  & 50 & 735 & 622 & 1.567 & 631 & 2.223 & 709 & 0.317 & 709 & 9.982\\
  & 60 & 1062 & 609 & 3.684 & 618 & 5.401 & 684 & 0.694 & 684 & 13.97\\
  & 70 & 1449 & 656 & 7.436 & 666 & 12.37 & 746 & 1.331 & 747 & 21.63\\
  & 80 & 1896 & 749 & 13.03 & 756 & 23.77 & 867 & 2.613 & 867 & 31.96\\
  & 90 & 2403 & 815 & 20.33 & 826 & 39.99 & 933 & 4.838 & 933 & 48.81 \\
  & 100 & 2970 & 810 & 30.35 & 823 & 66.04 & 951 & 12.50 & 952 & 78.62\\
  0.5 & 30 & 435 & 339 & 0.516 & 343 & 0.876 & 373 & 0.168 & 373 & 16.59\\
  & 40 & 780 & 411 & 1.456 & 418 & 2.386 & 464 & 0.474 & 464 & 16.21\\
  & 50 & 1225 & 466 & 4.159 & 473 & 7.550 & 534 & 1.177 & 535 & 22.34 \\ \bottomrule
\end{tabular}
 \caption{Bounds and computation times in seconds of $GL$, $MILP$, $LBB1$ and $LBB1^{skew}$ on Angle-distance instances. \label{BigAD}}
\end{table}

\section{Conclusion}
In this paper we consider the linearization problem for the \textsc{QCCP} and its applications.
We provide several sufficient conditions for linearizability, and show how these conditions can be used to obtain strong lower bounds for the \textsc{QCCP}.
The linearization based bound $LBB1$, resulting from the incident weak sum property,
is the most efficient LBB in terms of complexity and quality, see Table \ref{PreliminaryExperiments}.
We show here that the GL type bound for the \textsc{QCCP} also belongs to the family of linearization based bounds, see Theorem \ref{EquivalenceGLLBB},
by providing the appropriate sufficient condition, see Proposition \ref{GLsufficientcondition}.

The first level RLT bounds and/or the GL type bounds  are the only linearization based bounds for quadratic binary optimization problems that are implemented for various
binary quadratic optimization problems  up to date.
This paper shows that besides these two well-known bounds, the linearization based bounds introduced here are worth considering.

Here, we also present  how each sufficient condition can be used in an iterative bounding procedure. In particular, we introduce a new reformulation technique in which we search for the best equivalent representation of the residual cost matrix and its optimal linearizable matrix, see Algorithm \ref{AlgorithmLBB}.
We show how the resulting iterative procedure computes the best GL type bound in each iteration.
Our approach  outperforms known iterative bounding procedures that use the GL type bounds, see Table \ref{GLPreliminary}.
Moreover, we prove that the resulting linearization vectors in each step satisfy the constant value property, see Theorem \ref{CVP}.

Finally,  our numerical results show that  our approach outperforms several other bounds from the literature if we take into account both quality and efficiency. Although  the linearization based bounds $LBB1$  are dominated by the well known first level RLT bounds, they can be computed extremely fast.
For the Manhattan instances,  $LBB1$ bounds for instances up to 15000 arcs can be computed within 60 seconds.
However, other approaches fail to provide bounds for instances of this large sizes.

We expect that similar bounding procedures can be successfully applied for other quadratic optimization problems, such as the quadratic assignment problem,
the quadratic minimum spanning tree, and the quadratic traveling salesman problem.  However,  this is a subject of our future research.

\paragraph{Acknowledgements} We would like to thank two anonymous reviewers for their insightful comments and suggestions to improve an earlier version of this work.

\end{document}